\documentclass[12pt]{amsart}
\usepackage{latexsym}
\usepackage{amssymb}
\usepackage{amsxtra}
\usepackage{amsmath}
\usepackage{amsfonts}
\usepackage{epsfig}

\newcommand{\cleqn}{\setcounter{equation}{0}}
\newcommand{\clth}{\setcounter{theorem}{0}} 
\newcommand{\clfig}{\setcounter{figure}{0}} 
\newcommand {\sectionnew}[1]{\section{#1}\cleqn\clth\clfig}
\newcommand{\beq}{\begin{equation}}
\newcommand{\eeq}{\end{equation}}
\newcommand{\beqa}{\begin{eqnarray}}
\newcommand{\eeqa}{\end{eqnarray}}
\newcommand{\beaa}{\begin{eqnarray*}}
\newcommand{\ben}{\begin{eqnarray*}}
\newcommand{\eaa}{\end{eqnarray*}}
\newcommand{\een}{\end{eqnarray*}}

\newcommand \nc {\newcommand}
\numberwithin{equation}{section}

\newtheorem{theorem}{Theorem}[section]
\newtheorem{lemma}[theorem]{Lemma}
\newtheorem{proposition}[theorem]{Proposition}
\newtheorem{corollary}[theorem]{Corollary}
\newtheorem{definition}[theorem]{Definition}

\newtheorem{remark}[theorem]{Remark}
\newtheorem{conjecture}[theorem]{Conjecture}
\newtheorem{question}[theorem]{Question}

\newtheorem*{lemmaa}{Lemma A}
\newtheorem*{lemmab}{Lemma B}

\nc \thref{Theorem \ref}
\nc \leref{Lemma \ref}
\nc \prref{Proposition \ref}
\nc \coref{Corollary \ref}
\nc \deref{Definition \ref}
\nc \exref{Example \ref}
\nc \reref{Remark \ref}

\newcommand{\CC}{\mathbb{CP}^1}

\renewcommand{\P}{\mathcal{P}}
\newcommand{\Q}{\mathcal{Q}}

\newcommand{\W}{\mathcal{W}}
\newcommand{\A}{\mathcal{A}}
\newcommand{\B}{\mathcal{B}}
\newcommand{\C}{\mathbb{C}}
\newcommand{\D}{\mathcal{D}}

\newcommand{\F}{\mathcal{F}}
\renewcommand{\H}{\mathcal{H}}

\newcommand{\I}{\mathcal{I}}
\renewcommand{\L}{\mathcal{L}}
\newcommand{\M}{\mathcal{M}}

\renewcommand{\O}{\mathcal{O}}
\newcommand{\QQ}{\mathbb{Q}}

\newcommand{\T}{\mathcal{T}}

\newcommand{\Z}{\mathbb{Z}}

\newcommand{\bnu}{\overline{\nu}}
\newcommand{\bW}{\overline{W}}
\newcommand{\g}{\mathbf{g}}
\newcommand{\f}{\mathbf{f}}

\newcommand{\q}{\mathbf{q}}
\renewcommand{\t}{\mathbf{t}}
\newcommand{\x}{\mathbf{x}}
\newcommand{\y}{\mathbf{y}}

\def\res{\mathop{\rm res}\nolimits}

\def\d{\partial}
\def\ev{\mathop{\rm ev}\nolimits}

\def\diag{\mathop{\rm diag} \nolimits}

\def\iso{\cong}
\def\tensor{\otimes}

\def\Kahler{K\"ahler}
\def\Poincare{Poincar\'e}

\def\({\left(}
\def\){\right)}
\def\[{\left[}
\def\]{\right]}
\def\<{\left\langle}
\def\>{\right\rangle}

\def\gl{\lambda}
\def\ge{\epsilon}
\def\ga{\alpha}

\def\gb{\beta}

\newcommand{\stilde}{\tilde{s}}
\newcommand{\proj}{\mathbb{P}}

\begin{document}
\title[Orbifold structures and integrable hierarchies]
{Equivariant orbifold structures on the projective line and 
integrable hierarchies}

\author{Todor E. Milanov}
\address{Department of Mathematics\\ Stanford University\\ 
Stanford\\ CA 94305--2125\\ USA}
\email{milanov@math.stanford.edu}

\author{Hsian-Hua Tseng}
\address{Department of Mathematics\\ University of British Columbia\\ 
1984 Mathematics Road\\ Vancouver\\ B.C. V6T 1Z2\\ Canada}
\email{hhtseng@math.ubc.ca}

\date{\today}
\thanks{{\em 2000 Math. Subj. Class.} 14N35, 17B69, 32S30}
\thanks{
{\em Key words and phrases.} oscillating integrals, Frobenius structure,
orbifold quantum cohomology, bosonic Fock space, vertex operators, 
Hirota quadratic (bilinear) equations}

\begin{abstract}
Let ${\CC}_{k,m}$ be the orbifold structure on ${\CC}$ obtained 
via uniformizing the neighborhoods of 0 and $\infty$ respectively
by $z\mapsto z^k$ and $w\mapsto w^m.$ The diagonal action of the torus 
$T= \(S^1\)^2$ on ${\CC}$ induces naturally an action on the orbifold
${\CC}_{k,m}.$ In this paper we prove that if $k$ and $m$ are 
co-prime then Givental's prediction of the equivariant total 
descendent Gromov-Witten potential of ${\CC}_{k,m}$ satisfies certain
Hirota Quadratic Equations (HQE for short).
We also show that after an appropriate change of the variables, similar to
Getzler's change in the equivariant Gromov-Witten theory of $\CC$, the HQE 
turn into the HQE of the 2-Toda hierarchy, i.e.,  the 
Gromov-Witten potential of ${\CC}_{k,m}$ is a tau-function of 
the 2-Toda hierarchy. More precisely, we obtain a sequence of tau-functions
of the 2-Toda hierarchy from the descendent
potential via some translations. The later condition, that all
tau-functions in the sequence are obtained from a single one via 
translations, imposes a serious constraint on the solution of the 
2-Toda hierarchy. Our theorem leads
to the discovery of a new integrable hierarchy (we suggest to be 
called the Equivariant Bi-graded Toda Hierarchy), obtained from 
the 2-Toda hierarchy via a reduction similar to the one in \cite{Ge}. 
We conjecture that this new hierarchy governs, i.e., uniquely determines, the
equivariant Gromov-Witten invariants of $\CC_{k,m}.$
\end{abstract}
\maketitle

\sectionnew{Introduction}

Let $\CC_{k,m}$ be the orbifold structure on $\CC$ 
obtained via uniformizing the neighborhoods of 0 and $\infty$ respectively
by $z\mapsto z^k$ and $w\mapsto w^m.$ This uniformization induces naturally 
an orbifold structure on the hyperplane class bundle, such that the cyclic
groups $\Z_k$ and $\Z_m$ act trivially on the corresponding fibers.  The 
resulting orbifold bundle is denoted $\O^{\rm unif}(1).$

Let $T=S^1\times S^1$ and denote by $\nu_0$ and 
$\nu_1$ the characters of the representation dual to the standard
representation of $T$ in $\C^2.$ 
The $T$-equivariant cohomology of a point is naturally
identified with $\C[\nu_0,\nu_1].$ Furthermore, the diagonal
action of $T$ on $\C^2$ induces a $T$-action on 
$\C\mathbb{P}^1 = \(\C^2-\{0\}\)/\C^*$ and the later naturally induces a $T$-action on the orbifold $\CC_{k,m}.$  
We also equip the bundle $\O^{\rm unif}(1)$ 
with a $T$-action in such a way that the corresponding characters
on the fibers of $\O^{\rm unif}(1)$ at 0 and $\infty$ are respectively 
$\nu_0$ and $\nu_1.$ 

The equivariant orbifold cohomology $H$ of $\CC_{k,m}$ is by definition the 
equivariant cohomology of its inertia orbifold:
\ben
I\CC_{k,m}= \CC_{k,m}\  \sqcup \ 
\bigsqcup_{i=1}^{k-1} [{\rm pt}/\Z_k]\ \sqcup\ 
\bigsqcup_{j=1}^{m-1} [{\rm pt}/\Z_m],
\een 
where the orbifolds $[{\rm pt}/\Z_k]$ and $[{\rm pt}/\Z_m]$ are called 
{\em twisted sectors} and the torus $T$ acts trivially on them.     
We fix a basis in $H$:
\ben
{\bf 1}_{i/k} ,\ 1\leq i\leq k-1, & 
{\bf 1}_{0/k}= (p-\nu_1)/(\nu_0-\nu_1), \\
{\bf 1}_{j/m} ,\ 1\leq j\leq m-1, & 
{\bf 1}_{0/m} = (p-\nu_0)/(\nu_1-\nu_0),
\een 
where $p$ is the equivariant 1-st Chern class of $\O^{\rm unif}(1),$ 
${\bf 1}_{i/k}$ and ${\bf 1}_{j/m}$ are the units in the cohomologies of 
the corresponding twisted sectors and the indices $i/k$ and $j/m$ are 
identified respectively with elements in $\Z_k$ and $\Z_m$.  
Finally, let $(\ ,\ )$ be the equivariant orbifold {\Poincare} pairing in $H:$
\ben
\({\bf 1}_{0/k},{\bf 1}_{0/k}\)=1/(\nu_0-\nu_1),\quad 
\({\bf 1}_{i/k},{\bf 1}_{(k-i)/k}\)=1/k,\quad 1\leq i\leq k-1,\\
\({\bf 1}_{0/m},{\bf 1}_{0/m}\)=1/(\nu_1-\nu_0),\quad
\({\bf 1}_{j/m}, {\bf 1}_{(m-j)/m}\)=1/m,\quad 1\leq j\leq m-1,
\een 
and all other pairs of cohomology classes are orthogonal.  

By definition {\em the total descendent Gromov--Witten potential} 
of $\CC_{k,m}$ is
\ben
\D(\q) = \exp \left(\sum_{g,n,d} \ge^{2g-2}\frac{Q^d}{n!} 
\int_{[\overline{\M}_{g,n}(\CC_{k,m},d)]^{vir}} 
\prod_{a=1}^n\(\psi_a+\sum_{l=0}^\infty \ev_a^*(q_l)\psi_a^l \)\right),
\een
where $\overline{\M}_{g,n}(\CC_{k,m},d)$ is the moduli space of degree 
$d\in \Z$ stable holomorphic maps $f$ from a genus-$g$ 
Riemann surface, equipped with $n$ marked orbifold points, 
$\ev_a:\overline{\M}_{g,n}(\CC_{k,m},d)\rightarrow I\CC_{k,m}$ is the 
evaluation map at the $a$-th marked point, $\psi_a$ is the equivariant 1-st
Chern class of the line bundle on 
$\overline{\M}_{g,n}(\CC_{k,m},d)$ corresponding to the cotangent line at the $a$-th marked point,
$\q= \sum_{l=0}^\infty q_l z^l \in H[z]$, the integrals are performed against the {\em virtual fundamental classes} $[\overline{\M}_{g,n}(\CC_{k,m}, d)]^{vir}$, and the sum is over all non-negative
integers $g,n,d$ for which the moduli space 
$\overline{\M}_{g,n}(\CC_{k,m},d)$ is non-empty. 

The potential $\D$ is identified with an element of a bosonic Fock
space $\B$ which by definition is the vector space of functions on 
$H[z]$ which belong to the formal neighborhood of $-{\bf 1} \, z.$
Note that ${\bf 1} = {\bf 1}_{0/k}+{\bf 1}_{0/m}, $ therefore if we put
\ben
q_n = \sum_{i=0}^{k-1} q^{i/k}_n {\bf 1}_{i/k} + 
      \sum_{j=0}^{m-1} q^{j/m}_n {\bf 1}_{j/m},
\een 
then $\B$ is the set of formal series in the variables 
$q_n^{i/k}+\delta_n^1\delta_{0/k}^{i/k}  ,
 q_n^{j/m}+ \delta_n^1\delta_{0/m}^{j/m}$, whose coefficients are
formal Laurent series in $\ge.$ Here we used the Kronecker 
symbols: $\delta_a^b=1$ or $0$ depending whether $a=b$ or $a\neq b.$

We introduce the following {\em vertex operators} acting on the Fock space
$\B$:
\beq\label{vop_0}
\Gamma^{\pm } = \exp \(\pm 
\sum_{n\in \Z}\sum_{i=1}^k 
\frac{ \prod_{l=-\infty}^n (\nu + (-i/k+l)z)}
     { \prod_{l=-\infty}^0 (\nu + (-i/k+l)z)} \ 
\gl^{-(n+1)k+i} \ {\bf 1}_{(k-i)/k} \)\sphat
\eeq
where $\nu=(\nu_0-\nu_1)/k$, and the hat $\widehat{\quad}$ indicates the following
quantization rule. The exponent $\f^\pm$ of $\Gamma^\pm$ is written as
a product of two exponents: the first (left) one contains the summands with 
$n<0$ and the second (right) one with $n\geq 0$. Each summand 
corresponding to $n<0$ is expanded into a series of $z^{-1}.$  
The quantization rule consists of representing the terms 
$\phi_\ga (-z)^{-n-1}, n\geq 0$ and $\phi_\ga z^n, n\geq 0$
respectively by the operators of multiplication by 
the linear function $- \ge^{-1}\sum_\gb \eta_{\ga\gb}q_n^\gb$ and the 
differential operator $\ge \d/\d q_n^\ga.$ Here $\eta_{\ga\gb}= 
({\bf 1}_\ga,{\bf 1}_\gb)$ is the tensor of the {\Poincare} pairing.
Similarly, we introduce the vertex operator
$\overline{\Gamma}^{\pm}$ obtained from $\Gamma^\pm$ by switching 
$\nu_0\leftrightarrow\nu_1,$ and $k\leftrightarrow m.$

We say that a vector $\D$ in the Fock space $\B$ satisfies the Hirota quadratic equations (HQE) below if for each pair of integers $l$ and $n$
\beqa\notag
&&
\res_{\gl=\infty} 
\left(\gl^{n-l}\ \Gamma^{-}\tensor\Gamma^{+} - (Q/\gl)^{n-l}\ 
\overline{\Gamma}^{+}\tensor\overline{\Gamma}^{-}\right)\\ 
&&
\label{HQE}
\( e^{(n+1)\,\widehat{\bf 1}_{0/k} +n\,\widehat{\bf 1}_{0/m} }\tensor 
e^{
l\,\widehat{\bf 1}_{0/k}+(l+1)\,\widehat{\bf 1}_{0/m} } \) 
\(\D\tensor\D\)\ \frac{d\gl}{\gl} =0\ .
\eeqa
The HQE \eqref{HQE} are interpreted as follows. 
Switch to new variables $\x$ and $\y$ via the substitutions: $\q'=\x+\y$, 
$\q''=\x-\y$.
The LHS of the HQE expands as a series in $\y$ with coefficients Laurent 
series in $\gl^{-1}$, whose coefficients are quadratic polynomials in $\D,$ 
its partial derivatives and their translations. The residue is defined as 
the coefficient in front of $\gl^{-1}$.

Motivated by Givental's formula of the total descendent potential
of a {\Kahler} manifold with semi-simple quantum cohomology, we 
introduce an element $\D^{\rm Fr}$ of the Fock space of the following
type:
\beq\label{sashas_formula}
\D^{\rm Fr}= e^{{\rm F}^{(1)}(\tau)} \widehat{S}_\tau^{-1} 
\(\Psi_\tau R_\tau e^{U_\tau/z}\)\sphat \ \prod_{i=1}^{k+m} \D_{\rm pt}(\q^i).
\eeq
The different ingredients in this formula will be explained later. 
For a K\"ahler manifold equipped with a Hamiltonian torus action whose $0$ and $1$-dimensional strata are isolated, Givental \cite{G1, G3} proved that (\ref{sashas_formula}) agrees with the equivariant total descendant Gromov-Witten potential. His arguments, based on an ingenious localization analysis,  may be extended to orbifolds and, together with some new ingredients, used to prove that (\ref{sashas_formula}) agrees with the equivariant total descendant orbifold Gromov-Witten potential for a K\"ahler orbifold with a Hamiltonian torus action whose $0$ and $1$-dimensional strata are isolated. Details will be given in \cite{ts}.

Our goal here is to prove that the conjectural formula (\ref{sashas_formula}) has some very interesting property which in particular leads to the discovery of a new integrable hierarchy given in terms 
of HQE by \eqref{HQE}. 
Our main result can be stated this way.

\begin{theorem}\label{t1}The function 
$\D^{\rm Fr}$ satisfies \eqref{HQE}.
\end{theorem}

Let $y_1,y_2,\ldots $ and $ \overline y_1, \overline y_2,\ldots $ be 
two sequences of time variables related to $q_0^{i/k},q_{1}^{i/k},\ldots$ and 
$q_{0}^{j/m}, q_{1}^{j/m},\ldots$ via an upper-triangular 
linear change defined by the following relations:
{\allowdisplaybreaks
\beqa\label{change_qy}
\sum_{n\geq 0} (-w)^{-n-1}\frac{\d}{\d q_n^{i/k}} & = &
\sum_{n\geq 0} 
\frac{g_{i/k}}{\prod_{l=0}^{n}\(\nu -(l+i/k)w\)}\, 
\frac{\d}{\d y_{nk+i}}, \\
\label{change_qby}
\sum_{n\geq 0} (-w)^{-n-1}\frac{\d}{\d q_n^{j/m}} & = &
\sum_{n\geq 0} 
\frac{g_{j/m}}{\prod_{l=0}^{n}\(\bnu -(l+j/m)w\)}\, 
\frac{\d}{\d \overline{y}_{nm+j}},
\eeqa}
where $n\geq 0$, $1\leq i\leq k$, $1\leq j\leq m,$ 
$\bnu = (\nu_1-\nu_0)/m$ and  
\ben
g_\ga:=\({\bf 1}_\ga,{\bf 1}_{-\ga}\),\quad \ga\in \Z_k\sqcup\Z_m.
\een

\begin{theorem}\label{t2} 
Let $\D_n(\q) = Q^{n^2/2}\D^{\rm Fr}(\q +n\ge\,{\bf 1}).$ Then
the changes \eqref{change_qy}--\eqref{change_qby} transforms $\{\D_n\}$ 
into a sequence of tau-functions of the 2-Toda hierarchy. 
\end{theorem}

\medskip

Recall that the KdV hierarchy is a reduction of the KP hierarchy which
in terms of tau-functions can be described as follows: tau-functions of KdV hierarchy are tau-functions of
KP which depend only on odd variables. In our case we have a reduction
of the 2-Toda hierarchy which in terms of tau-functions can be 
described as sequences of tau-functions of 2-Toda obtained from a single function by some translations. In Appendix \ref{2toda} we describe what kind of constraint the later condition imposes on the Lax operators of 2-Toda.

\subsection*{Acknowledgments} 
We are thankful to B. Dubrovin for showing interest in our work and for
pointing out that the 2-Toda hierarchy is too big to govern the 
Gromov-Witten theory of the orbifold $\CC_{k,m}.$ This made us realize
that our HQEs describe a reduction of the 2-Toda hierarchy. 
Many thanks to the organizers M. Bertola and D. Korotkin of the conference 
``Short program on Moduli spaces of Riemann surfaces and related topics'', held in  Montreal, 
Canada during June 3 -- 15, 2007, where the first author was given the 
opportunity to lecture on some of the results in this article. The second author is grateful to Institut Mittag-Leffler (Djursholm, Sweden) for hospitality and support during his visit to the program ``moduli spaces''.


\sectionnew{Proof of \thref{t2}}


Let $h_l(x_1,\ldots,x_n)$ and $e_l(x_1,\ldots,x_n)$ be the symmetric
polynomials of degree $l$ defined as follows:
{\allowdisplaybreaks
\ben
\prod_{i=1}^n(1+tx_i) & = & \sum_{l\geq 0} t^le_l(x_1,\ldots,x_n), \\
\prod_{i=1}^n \frac{1}{1+t x_i} & = & \sum_{l\geq 0} t^lh_l(x_1,\ldots,x_n).
\een }
To avoid cumbersome notations we put
\ben
\delta_{kN+i}:=\frac{g_{i/k}}{(N+i/k)! }\frac{\d}{\d y_{kN+i}} \quad 
\overline\delta_{mN+j}:=\frac{g_{j/m}}{(N+j/m)! }
\frac{\d}{\d \overline{y}_{mN+j}},
\een
where $N\geq 0,$ $1\leq i\leq k,$ $1\leq j\leq m,$ and for
a positive real number $\ga\notin\Z$ we put 
$\ga! = \{\ga\}(\{\ga\}+1)\ldots \ga$ where $\{\ga\}$ is 
the fractional part of $\ga.$ Note that the change of variables 
can be written as follows
\ben
\frac{\d}{\d q_n^{i/k}} = \sum_{N=0}^n \nu^{n-N}
h_{n-N}\(\frac{1}{i/k}, \frac{1}{i/k+1},\ldots, \frac{1}{i/k+N}\)\,
\delta_{kN+i}\ ,
\een 
\ben
\frac{\d}{\d \overline q_n^{j/m}} = \sum_{N=0}^n \bnu^{n-N}
h_{n-N}\(\frac{1}{j/m}, \frac{1}{j/m+1},\ldots, \frac{1}{j/m+N}\)\,
\overline\delta_{mN+j}\ .
\een 
Following an argument of E. Getzler (\cite{Ge}, Proposition A.1) we 
show that the above formulas can be inverted. Namely,
\begin{lemma}\label{inv}
 The following formulas hold
\ben
\delta_{kL+i} =\sum_{n=0}^L \nu^{L-n}
e_{L-n}\(\frac{1}{i/k}, \frac{1}{i/k+1},\ldots, \frac{1}{i/k+L-1}\)\, 
\frac{\d}{\d q_n^{i/k}},
\een
\ben
\overline\delta_{mL+j} =\sum_{n=0}^L \bnu^{L-n}
e_{L-n}\(\frac{1}{j/m}, \frac{1}{j/m+1},\ldots, \frac{1}{j/m+L-1}\)\, 
\frac{\d}{\d \overline q_n^{j/m}}.
\een
\end{lemma}
\proof
We prove the first identity. The argument for the second one is similar.
We need to show that the following identity holds for any two integers 
$L\geq N:$ 
\ben
\sum_{n=N}^L \nu^{L-n}\nu^{n-N} 
e_{L-n}\(\frac{1}{i/k}, \frac{1}{i/k+1},\ldots, \frac{1}{i/k+L-1} \)\times \\
h_{n-N}\(\frac{1}{i/k}, \frac{1}{i/k+1},\ldots, \frac{1}{i/k+N}\) =\delta_N^L.
\een
If $L=N$ then the identity is obviously true. Assume that $L>N$. Then  
the LHS can be interpreted as the coefficient in front of $\nu^{L-N}$ in 
the product:
\ben
\prod_{a=0}^{L-1} (1+\nu/(i/k+a)) \ 
\prod_{a=0}^N \frac{1}{(1+\nu/(i/k+a))}.
\een
However, with respect to $\nu$, this is a polynomial of degree $L-N-1.$ 
\qed

The proof of \thref{t2} amounts to changing the variables in the vertex
operators $\Gamma^\pm$ and $\overline{\Gamma}^{\pm}.$ Let us begin
with $\Gamma^\pm$ and more precisely with the summands in \eqref{vop_0}
corresponding to $n> 0$ and $i=k-j$, $1\leq j\leq k-1.$ The 
coefficient in front of $\gl^{-nk-j}$ transforms as follows:
\ben
&&
\({\bf 1}_{j/k} (\nu+(j/k)z)\ldots(\nu+(j/k+ n-1)z)\)\sphat = \\
&&
=(n-1+j/k)! \sum_{l=0}^n \(z^l{\bf 1}_{j/k}\)\sphat 
\nu^{n-l}e_{n-l}\(\frac{1}{j/k}, \frac{1}{j/k+1},\ldots, \frac{1}{j/k+n-1}\) =\\
&&
=(n-1+j/k)!\,\ge\delta_{kn+j} = \frac{1}{kn+j}\,\ge \d_{y_{kn+j}},
\een   
We used that $\(z^l {\bf 1}_{j/k}\)\sphat = \ge\d/\d q_l^{j/k}$ and the first
identity in \leref{inv}. Similarly, one can verify that the above answer
is valid also for all pairs $n,i$ such that either $n>0$ and $i=k,$ or $n=0$ and 
$1\leq i\leq k-1.$ 

Let $y_{Nk+i} = \sum_{L\geq 0} a_{N,L} q_L^{i/k}$ be a linear change. 
Then by the chain rule: $\d_{q_L^{i/k}} = \sum_{N\geq 0} a_{N,L} \d_{y_{Nk+i}}$, i.e.,
\ben
\sum_{L\geq 0} (-w)^{-L-1} \d_{q_L^{i/k}} = \sum_{N\geq 0}
\(\sum_{L\geq 0} a_{N,L}(-w)^{-L-1} \) \d_{y_{Nk+i}}.
\een 
On the other hand, since our linear change is defined by \eqref{change_qy}, we get  
\ben
\sum_{L\geq 0} a_{N,L}(-w)^{-L-1} = \frac{g_{i/k}}{\prod_{l=0}^N (\nu -(l+i/k)w)}.
\een
Note that with respect to
the {\Poincare } pairing we have ${\bf 1}_{(k-i)/k} = g_{i/k} {\bf 1}^{i/k}.$
The term corresponding to $n=-N-1<0$ and $i,$ $1\leq i\leq k$, in the 
exponent of $\Gamma^+$ transforms as follows:
\ben
\(\frac{g_{i/k}}{\prod_{l=0}^N (\nu -(l+i/k)z)} {\bf 1}^{i/k}\)\sphat \gl^{Nk+i} = 
\sum_{L\geq 0} a_{N,L}\((-z)^{-L-1}{\bf 1}^{i/k} \)\sphat \gl^{Nk+i}. 
\een 
On the other hand, according to our quantization rules, 
$\((-z)^{-L-1}{\bf 1}^{i/k} \)\sphat= -\ge^{-1}q_L^{i/k}.$ Thus the term 
corresponding to $n=-N-1$ and $i$ is $-\ge^{-1}y_{Nk+i}\gl^{Nk+i}.$

Finally, the term corresponding to $n=0$ and $i=k$ is $\widehat{\bf 1}_{0/k}.$
Thus, in the new coordinates, the vertex operators $\Gamma^\pm$ are given by
\ben
\Gamma^{\pm} = \exp\(\mp\sum_{n=1}^\infty \gl^n y_n/\ge\)
\exp\(\pm \sum_{n=1}^\infty \frac{\gl^{-n}}{n}\, \ge\, \d_{y_n}\) 
e^{\pm\widehat{\bf 1}_{0/k}}. 
\een  
Similarly the other two vertex operators $\overline{\Gamma}^\pm$ are given by
\ben
\overline{\Gamma}^\pm = 
\exp\(\mp\sum_{n=1}^\infty \gl^n \overline y_n/\ge\)
\exp\(\pm \sum_{n=1}^\infty \frac{\gl^{-n}}{n}\, \ge\, \d_{\overline y_n}\) 
e^{\pm\widehat{\bf 1}_{0/m}}. 
\een  
Substitute these formulas in the HQE in \thref{t1} and note that by definition:
$\D_n = Q^{n^2/2} e^{n\,( \widehat{\bf 1}_{0/k}+\widehat{\bf 1}_{0/m})}
\D^{\rm Fr}.$  
After a short simplification and up to rescaling $y_n$ and 
$\overline y_n$ by $\ge^{-n}$ we get the HQE of the 2-Toda 
hierarchy (see appendix \ref{2toda}).


\sectionnew{Gromov-Witten theory of $\CC_{k,m}$}


\subsection{The system of quantum differential equations}

For some basics on orbifold Gromov-Witten theory we refer the 
reader to \cite{CR} and \cite{ab-gr-vi1, ab-gr-vi2}. We recall the vector space $H$ which by
definition coincides with the vector space of the equivariant 
cohomology algebra of the inertia orbifold $I\CC_{k,m}.$ For 
each $\tau\in H$ the orbifold quantum cup product $\bullet_\tau$
is a commutative associative multiplication in $H$ defined by
the following genus-0 Gromov-Witten invariants: 
\ben
({\bf 1}_\ga\bullet_\tau{\bf 1}_\gb,{\bf 1}_\gamma) = 
\sum_{l,d\geq 0} 
\frac{Q^d}{l!} \int_{[\overline{\M}_{0,l+3}(\CC_{k,m};d)]^{vir}}
{\rm ev}^*\({\bf 1}_\ga\tensor {\bf 1}_\gb \tensor {\bf 1}_\gamma \tensor
\tau^{\tensor l} \),
\een  
where ${\rm ev}$ is the evaluation map at the $l+3$ marked points. 
For brevity the RHS of the above equality will be denoted by 
the correlator $\<{\bf 1}_\ga,{\bf 1}_\gb,{\bf 1}_\gamma\>_{0,3}(\tau).$ We use
similar correlator notations for the other Gromov-Witten invariants as well. 

It is a basic fact in quantum cohomology theory that the 
following system of ordinary differential equations is compatible:
\ben
z\d_{\tau^\ga} \Phi = {\bf 1}_\ga\bullet_\tau\, \Phi,\quad
\ga\in \Z_k\sqcup\Z_m,
\een
where $\tau^\ga$ are the coordinates of $H$ with respect to the 
basis ${\bf 1}_\ga,$ and ${\bf 1}_\ga\bullet_\tau$ is the operator of
quantum multiplication by ${\bf 1}_\ga.$ This system is called the
system of Quantum Differential Equations (QDE) of the 
orbifold $\CC_{k,m}.$ 

If the parameter $z$ is close to $\infty$ then the following 
${\rm End}(H)$-valued series $S= 1+ S_1z^{-1}+\ldots $ provides a 
fundamental solution to the system of QDE:
\ben
(S_\tau\phi_\ga,\phi_\gb) = (\phi_\ga,\phi_\gb) + 
\sum_{k=0}^\infty \< \psi^k\phi_\ga,\phi_\gb\>_{0,2}(\tau)z^{-k-1}.
\een  
Our first goal is to explicitly compute $S_\tau$ for $\tau\in H^2(\CC_{k,m}).$

\subsection{The J-function}
The idea is to compute the so called $J$-function of $\CC_{k,m}$ defined 
by
\ben
J_{\CC_{k,m}}(\tau) = z\,{\bf 1} + \tau + \sum_{k=0}^\infty
\<{\bf 1}_\ga\psi^k\>_{0,1}(\tau) {\bf 1}^\ga z^{-k-1}. 
\een
In this Section we calculate the restrictions to $H^2(\CC_{k,m})$ of the $J$-function and its partial derivatives. Due to technical reasons we assume that $k,m$ are co-prime. However we conjecture that the main results, Proposition \ref{J}, Corollaries \ref{qde} and \ref{dJ} also hold in general. The general case will be addressed elsewhere using results of \cite{ccit} concerning toric Fano stacks.

Suppose that $k,m$ are co-prime. Then it is easy to see that $\CC_{k,m}$ is isomorphic to the weighted projective line $\proj (k,m)$, which is defined to be the stack quotient $[(\C^2- 0)/\C^*]$ under the following $\C^*$-action: $$\lambda\cdot (z_0,z_1)=(\lambda^{-k}z_0,\lambda^{-m}z_1).$$ 
It is important to note that the identification of the isotropy groups at stacky points are different. For $\proj(k,m)$, there is a natural map $\proj(k,m)\to \proj^1$ given by $[z_0;z_1]\mapsto [z_0^m;z_1^k]$. The neighborhood $\{z_0^m\neq 0\}\subset \proj^1$ of $0=[1;0]$ has the coordinate $z_1^k/z_0^m$, and the stack structure over $0$ is given by $z_1/z_0^{m/k}\mapsto z_1^k/z_0^m$ where $\Z_k$ acts by multiplication by $\exp(\frac{-2\pi\sqrt{-1}m}{k})$, while in case of $\CC_{k,m}$, $\Z_k$ acts by multiplication by $\exp(\frac{2\pi\sqrt{-1}}{k})$. The neighborhood $\{z_1^k\neq 0\}\subset \proj^1$ of $\infty=[0;1]$ has the coordinate $z_0^m/z_1^k$, and the stack structure over $\infty$ is given by $z_0/z_1^{k/m}\mapsto z_0^m/z_1^k$ where $\Z_m$ acts by multiplication by $\exp(\frac{-2\pi\sqrt{-1}k}{m})$, while in case of $\CC_{k,m}$, $\Z_m$ acts by multiplication by $\exp(\frac{2\pi\sqrt{-1}}{m})$.

The standard $T=(S^1)^2$ action on $\C^2$ descends to a $(S^1)^2$ action on $\proj (k,m)$. This gives a $(S^1)^2$-action on the line bundle $\O_{\proj (k,m)}(1)$. Let $\lambda_0/k$ and $\lambda_1/m$ be the weights of this action at $0$ and $\infty$ respectively.

\begin{definition}
For each real number $r$ we denote by $\{r\}\in(0,1]$ the unique real
number s.t. $r- \{r\}\in\Z.$ Note the range of $\{r\}$.
\end{definition}
 
\begin{proposition}\label{J}
The $T$-equivariant $J$-function of $X=\CC_{k,m}$ is given by the following formula
\begin{equation}\label{J-func}
\begin{split}
ze^{\tau\nu_0/z}\ 
\sum_{d\in \mathbb{Z}_{\geq 0}}
\frac{Q^{dm}e^{dm\tau}}{
d!z^d
\prod_{b=\{dm/k\}}^{dm/k}(\nu+bz) } {\bf 1}_{-dm/k}\  + \\
+
\ ze^{\tau\nu_1/z}\ 
\sum_{d\in\mathbb{Z}_{\geq 0}}
\frac{Q^{dk}e^{dkt}}{
\prod_{b=\{dk/m\}}^{dk/m}(\bnu +bz)
d!z^d} {\bf 1}_{-dk/m}.
\end{split}
\end{equation}
where  if $d=0$ then both fractions are by definition 
$1$ and in each product, $b$ varies over all rational numbers which have the
same fractional part as the corresponding upper (or lower) range of the product.
\end{proposition}
\proof
The calculation of \cite{cclt}, which is easily seen to apply to the $T$-equivariant setting, yields the following formula for the $T$-equivariant $J$-function for $\proj(k,m)$:
{\allowdisplaybreaks
\begin{equation}\label{J_Pmk}
\begin{split}
&J_{\proj(k,m)}(t)\\
&=ze^{Pt/z}\left({\bf 1}_0\sum_{d\in \mathbb{Z}_{\geq 0}}\frac{Q^de^{dt}}{\prod_{b_0=1}^{dk}(c_1^T(\O_{\proj (k,m)}(k))-\lambda_0+b_0z)\prod_{b_1=1}^{dm}(c_1^T(\O_{\proj (k,m)}(m))-\lambda_1+b_1z)}\right.\\
&+\sum_{i=1}^{k-1}{\bf 1}_{i/k}\sum_{d\in \mathbb{Z}_{\geq 0}}\frac{Q^{d+i/k}e^{(d+i/k)t}}{\prod_{b_0=1}^{kd+i}(c_1^T(\O_{\proj (k,m)}(k))-\lambda_0+b_0z)\prod_{b_1=\{im/k\}}^{dm+im/k}(c_1^T(\O_{\proj (k,m)}(m))-\lambda_1+b_1z)}\\
&\left.+\sum_{j=1}^{m-1}{\bf 1}_{j/m}\sum_{d\in\mathbb{Z}_{\geq 0}}\frac{Q^{d+j/m}e^{(d+j/m)t}}{\prod_{b_0=\{jk/m\}}^{dk+jk/m}(c_1^T(\O_{\proj (k,m)}(k))-\lambda_0+b_0z)\prod_{b_1=1}^{dm+j}(c_1^T(\O_{\proj (k,m)}(m))-\lambda_1+b_1z)}\right).
\end{split}
\end{equation}}
Here $P=c_1^T(\O_{\proj (k,m)}(1))\in H^2(\proj(k,m))$ and $t$ is its coordinate. 
(\ref{J-func}) follows from (\ref{J_Pmk}) by incorporating the following changes:
\begin{itemize}
\item
In our notation, ${\bf 1}_0={\bf 1}_{0/k}+{\bf 1}_{0/m}$.
\item
We want to measure degrees of curve classes using 
$\O^{\text{unif}}(1)=\O_{\proj(k,m)}(km)$, where in \cite{cclt} 
$\O_{\proj (k,m)}(1)$ is used. As a consequence the degree of a 
curve class we want to use is $km$ times theirs. 
\item
We use the coordinate $\tau$ of the class $c_1^T(\O^{\text{unif}}(1))$ 
as the variable for $J_X$.
\item
We have the following equalities:
\begin{equation*}
\begin{split}
c_1^T(\O_{\proj (k,m)}(k))=p/m, &\quad c_1^T(\O_{\proj (k,m)}(m))=p/k;\\
p\cdot {\bf 1}_{i/k}= p|_0=m\lambda_0, &\quad p\cdot {\bf 1}_{j/m}=p|_\infty=k\lambda_1;\\
\nu_0=m\lambda_0, &\quad \nu_1=k\lambda_1.
\end{split}
\end{equation*}
Using this we rewrite 
\begin{equation*}
\begin{split}
&\frac{1}{\prod_{1\leq b\leq B}(c_1^T(\O_{\proj (k,m)}(k))-\lambda_0+bz)}{\bf 1}_{i/k}=
\frac{1}{B!z^B}{\bf 1}_{i/k},\\
&\frac{1}{\prod_{1\leq b\leq B}(c_1^T(\O_{\proj (k,m)}(k))-\lambda_0+bz)}{\bf 1}_{j/m}=
\frac{1}{\prod_{1\leq b\leq B}(\bnu+bz)}{\bf 1}_{j/m},\\
&\frac{1}{\prod_{1\leq b\leq B}(c_1^T(\O_{\proj (k,m)}(m))-\lambda_1+bz)}{\bf 1}_{j/m}=
\frac{1}{B!z^B}{\bf 1}_{j/m},\\
&\frac{1}{\prod_{1\leq b\leq B}(c_1^T(\O_{\proj (k,m)}(m))-\lambda_1+bz)}{\bf 1}_{i/k}=
\frac{1}{\prod_{1\leq b\leq B}(\nu+bz)}{\bf 1}_{i/k}.
\end{split}
\end{equation*}
\item
The difference in the identification of isotropy groups discussed above 
imposes the changes ${\bf 1}_{i/k}\mapsto {\bf 1}_{-im/k}$ and 
${\bf 1}_{j/m}\mapsto {\bf 1}_{-kj/m}$.
\end{itemize}
\qed

A direct calculation gives the following
\begin{corollary}\label{qde}
The $J$-function $J_X$ satisfies the following differential equation:
\begin{equation}
\prod_{i=0}^{k-1}
\left(\frac{z}{m}\frac{\partial}{\partial\tau}-
\frac{\nu_0}{m}-iz\right)
\prod_{j=0}^{m-1}
\left(\frac{z}{k}\frac{\partial}{\partial \tau}-
\frac{\nu_1}{k}-jz\right) J_X=
Q^{km}e^{km\tau}J_X
\end{equation}
\end{corollary}

\begin{corollary}\label{dJ}
The restrictions of the partial derivatives of $J_X$ to the small
parameter space are given as follows:
{\allowdisplaybreaks
\begin{equation}\label{dJ_ik}
\begin{split}
(kg_{i/k})^{-1}z\d_{i/k}J_X &= 
ze^{\tau\nu_0/z}\ 
\sum_{d\in \mathbb{Z}_{\geq 0}}
\frac{Q^{dm}e^{dm\tau}}{
d!z^d}
\frac{\prod_{b<\{(dm-i)/k\}} (\nu+bz) } 
{
\prod_{b\leq (dm-i)/k}(\nu+bz) } 
{\bf 1}_{(-dm+i)/k}\   \\
&+
\ ze^{\tau\nu_1/z}\ 
\sum_{d\in\mathbb{Z}_{\geq 0}}
\frac{Q^{dk+i}e^{(dk+i)t}}{
\prod_{b=\{(dk+i)/m\}}^{(dk+i)/m}(\bnu +bz)
d!z^d} {\bf 1}_{-(dk+i)/m},\quad 1\leq i\leq k; 
\end{split}
\end{equation}}
{\allowdisplaybreaks
\begin{equation}\label{dJ_jm}
\begin{split}
(mg_{j/m})^{-1}z\d_{j/m}J_X  &= 
ze^{\tau\nu_0/z}\ 
\sum_{d\in \mathbb{Z}_{\geq 0}}
\frac{Q^{dm+j}e^{(dm+j)\tau}}{
d!z^d
\prod_{b=\{(dm+j)/k\}}^{(dm+j)/k}(\nu+bz) } 
{\bf 1}_{-(dm+j)/k}\  \\
&+
\ ze^{\tau\nu_1/z}\ 
\sum_{d\in\mathbb{Z}_{\geq 0}}
\frac{Q^{dk}e^{(dk)t}}
{d!z^d}
\frac{\prod_{b<   \{(dk-j)/m\}}(\bnu +bz) } 
     {\prod_{b\leq {(dk-j)/m }}(\bnu +bz)} 
{\bf 1}_{(-dk+j)/m},\quad 1\leq j\leq m,
\end{split}
\end{equation}}
where the notations and the conventions are the same as above.
\end{corollary}
The idea of the proof, borrowed from Section 5 of \cite{cclt}, is 
to express 
the partial derivatives of the $J$-function as linear combinations 
of derivatives along $H^2(\CC_{k,m})$. This is possible only when $k$ and $m$ 
are co-prime. The computation is straightforward but a bit
technical. It will be given in Appendix \ref{comb_setup}.

\subsection{Equivariant quantum cohomology of $\CC_{k,m}$ }

Put $N:=k+m.$ 
Let $k$ and $m$ be co-prime numbers. Then as it was explained above, 
$\CC_{k,m}$ is isomorphic as an orbifold to the weighted projective line.
We recall \cite{cclt} Corollary 1.2. The proof of this corollary (see 
\cite{cclt} section 5) generalizes to equivariant settings and we get 
the following description of the equivariant quantum cup product of 
$\CC_{k,m}$ at a point $\tau=t_N p,\ t_N\in \C.$ The map
\ben
&
{\bf 1}_{i/k}\mapsto \phi_{i/k}:= x^i,\quad 1\leq i\leq k-1, & 
{\bf 1}_{0/k}\mapsto \phi_{0/k}:= k x^k/(\nu_0-\nu_1), \\
& {\bf 1}_{j/m}\mapsto \phi_{j/m}:=y^j,\quad 1\leq j\leq m-1, &
{\bf 1}_{0/m}\mapsto \phi_{0/m}:= my^m/(\nu_1-\nu_0), 
\een
identifies the algebra $(H,\bullet_\tau),$ with the algebra 
\ben
\C[x,x^{-1}]/\< \d_x f \>, 
\een
where $f=x^k+(Qe^{t_N}/x)^m + \nu_1\log x + \nu_0 \log (Qe^{t_N}/x).$ 

The description of small orbifold quantum cohomology of $\CC_{k,m}$ is in fact valid without assuming that $k,m$ are co-prime. This may be seen as follows. Put $x_i={\bf 1}_{i/k}, y_j={\bf 1}_{j/m}$. Then in equivariant orbifold cohomology it is easy to see that 
\begin{equation*}
\begin{split}
x_i\cdot y_j=0, &\quad i,j\neq 0;\\
x_{i_1}\cdot x_{i_2}=x_{i_1+i_2}, &\quad i_1+i_2\leq k-1;\\
y_{j_1}\cdot y_{j_2}=y_{j_1+j_2}, &\quad j_1+j_2\leq m-1.
\end{split}
\end{equation*}
Also,
\begin{equation*}
\begin{split}
x_{k-1}\cdot x&=\langle x_{k-1},x,1\rangle_{0,3,0}P.D.(1)+\langle x_{k-1},x,p\rangle_{0,3,0}P.D.(p)\\
&=p/k+\nu_0/k
\end{split}
\end{equation*}
since $\langle x_{k-1},x,1\rangle_{0,3,0}=1/k, \langle x_{k-1},x,p\rangle_{0,3,0}=\int_{B\mathbb{Z}_k}p=\nu_0/k$, and $P.D.(1)=p, P.D.(p)=1$. Similarly we have $y_{m-1}\cdot y=p/m+\nu_1/m$. So the equivariant orbifold cohomology algebra can be identified with $$\C [\nu_0,\nu_1][x,y]/(kx^k-\nu_0=my^m-\nu_1, xy=0),$$ where $x:=x_1, y:=y_1$.

To calculate the small equivariant orbifold quantum cohomology we only need to find the correct deformations of the two relations $kx^k-\nu_0=my^m-\nu_1, xy=0$. We will use the fact that the small equivariant orbifold quantum cohomology algebra is {\em graded} as a $\C$-algebra, with $\text{deg}\,x=1/k, \text{deg}\,y=1/m, \text{deg}\,\nu_0=\text{deg}\,\nu_1=1, \text{deg}\, q=1/k+1/m$. By degree reason it is easy to see that the relation $xy=0$ is deformed to $xy=q$. The relation $kx^k-\nu_0=my^m-\nu_1$ remains undeformed. This can be seen in the same way as its non-equivariant counterpart treated in \cite{MT}, Section 4.3. Thus the small equivariant orbifold quantum cohomology of $\CC_{k,m}$ is isomorphic to $$\C[[q]][\nu_0,\nu_1][x,y]/(kx^k-\nu_0=my^m-\nu_1, xy=q).$$
Relationship between small quantum cohomology and big quantum cohomology restricted to $H^2$ imposes the change of variable $q=Qe^{t_N}$. This yields the description above.

We conjecture that the full equivariant quantum cohomology can be described in a similar
way. Namely, let $M$ be the family of functions on the complex circle $\C^*$
of the type:
\ben
f_t = x^k +\sum_{i=1}^k t_i x^{k-i} + 
\sum_{j=1}^{m-1} t_{k+j}\Big(\frac{Qe^{t_N}}{x}\Big)^j +\Big(\frac{Qe^{t_N}}{x}\Big)^m+
\nu_1\log x + \nu_0 \log\(\frac{Qe^{t_N}}{x}\).  
\een
Each tangent space of $M$ is equipped with an algebra structure via
the map:
\ben
T_tM \iso \C[x,x^{-1}]/\< \d_x f_t \> ,\quad \d/\d t_i\mapsto \[\d f_t/\d t_i\],\ 
1\leq i\leq N.
\een
Let $\omega:=dx/x$ be the standard volume form on $\C^*$. We equip each tangent
space $T_tM$ with a residue metric:
\ben
([\phi_1],[\phi_2])_t =-(\res_{x=0}+\res_{x=\infty}) 
\frac{\phi_1\omega\ \phi_2\omega }{df_t}.
\een
We claim that this is a flat metric on $M$ and we prove it by constructing
explicitly a coordinate system on $M$ such that the metric is constant.
If $x$ is close to $\infty$ then the equation 
\beq\label{change_infty}
f_t(x) = \gl^k + \nu_1\log \gl + \nu_0\log \( Q/\gl\)
\eeq
admits a unique solution of the type $x=\gl + a_0(t) + a_1(t)\gl^{-1}+\ldots$,
i.e., the equation determines a coordinate change near $x=\infty$ and 
we have the following expansion
\ben
\log x = \log \gl - \frac{1}{k} \Big( \sum_{i=1}^k \tau^{i/k} \gl^{-i} \Big) +
O(\gl^{-k-1}),
\een
where $\tau^{i/k}$ are polynomials in $t=(t_1,t_2,\ldots,t_N).$ More precisely,
by using 
\beq\label{change_explicit}
(i/k)\tau^{i/k} = -\res_{x=\infty} \gl^i \omega ,\ 1\leq i\leq k,
\eeq
we get 
\ben
\tau^{1/k}& =& t_1,\\
\tau^{i/k}& =& t_i + f_{i/k}(t_1,\ldots,t_{i-1}),\quad 2\leq i\leq k-1, \\
\tau^{0/k}& =& t_k + \nu_0 t_N,
\een
where $f_{i/k}$ are polynomials in $t_1,\ldots,t_{i-1}$ of degrees $\geq 2$. 
They can be computed explicitly by taking the coefficient in front of $x^{-i}$ 
in the following Laurent polynomial:
\ben
\frac{1}{i/k}\,
\sum_{n=2}^{i}{i/k \choose n} 
\Big( \frac{t_1}{x}+\ldots + \frac{t_{i-1}}{x^{i-1}} \Big)^n.  
\een
This formula is obtained from formula \eqref{change_explicit} by truncating 
the terms in the change \eqref{change_infty} that do not contribute to the 
residue in \eqref{change_explicit}. 

The rest of the flat coordinates can be constructed in a similar way. 
Let $y=Qe^{t_N}/x$ be another coordinate on the complex circle. Then 
each $f_t\in M$ assumes the form:
\ben
y^m + \sum_{j=1}^{m} t_{k+m-j} y^{m-j} + 
\sum_{i=1}^{k-1} t_{k-i}\Big(\frac{Qe^{t_N}}{y}\Big)^{i} +\Big(\frac{Qe^{t_N}}{y}\Big)^k+
\nu_0\log y + \nu_1 \log\(\frac{Qe^{t_N}}{y}\).
\een 
If $y$ is close to $\infty$ then the equation 
\beq\label{change_0}
f_t(y) = \gl^m + \nu_0\log \gl + \nu_1\log \( Q/\gl\)
\eeq
determines a new coordinate near $y=\infty$ and we have the following
expansion:
\ben
\log y = \log \gl - \frac{1}{m} \Big( \sum_{j=1}^m \tau^{j/m} \gl^{-j} \Big) +
O(\gl^{-m-1}),
\een
where $\tau^{j/m}$ are polynomials in $t=(t_1,t_2,\ldots,t_N).$ The same 
arguments as above yield the following: 
\ben
\tau^{1/m}& =& t_{k+m-1},\\
\tau^{j/m}& =& t_{k+m-j} + f_{j/m}(t_{k+m-1},\ldots,t_{k+m-(j-1)}),\quad 2\leq j\leq m-1, \\
\tau^{0/m}& =& t_k + \nu_1 t_N,
\een
where $f_{j/m}$ are polynomials of degrees at least $2$ and can be computed
explicitly by taking the coefficient in front of $y^{-j}$ of the following 
Laurent polynomial:
\ben
\frac{1}{j/m}\,
\sum_{n=2}^{j}{j/m \choose n} 
\Big( \frac{t_{k+m-1}}{y}+\ldots + \frac{t_{k+m-(j-1)}}{y^{j-1}} \Big)^n.  
\een
\begin{lemma}\label{res_poincare_pairing}
In the coordinate system $\{\tau^\ga\}_{\ga\in \Z_k\sqcup\Z_m}$, the residue
pairing coincides with the {\Poincare} pairing. More precisely:
\ben
(\d/\d \tau^\ga,\d/\d \tau^\gb) = ({\bf 1}_\ga, {\bf 1}_\gb).
\een
\end{lemma}
\proof
We prove the equality only when $\ga=i/k,\gb=i'/k, 1\leq i,i'\leq k.$
The other cases may be treated in a similar way. Let us compute the residue at
$x=\infty$ in the residue pairing. We change from $x$ to the coordinate
$\gl$ defined by equation \eqref{change_infty}. Differentiation by parts 
yields  $\d_{\tau^\ga}f_t + f_t'\,(\d_{\tau^\ga}x) = 0.$ Therefore
\ben
\d_{\tau^\ga}f_t \omega =-\(\d_{\tau^\ga}\log x\)\, df_t = 
k^{-1}\( \gl^{-i} + O(\gl^{-k-1})\)
\(k\gl^k + \nu_1-\nu_0\)\frac{d\gl}{\gl}.
\een
Now, the $\(-\res_{x=\infty}\)$-term in the residue pairing 
of $(\d/\d \tau^\ga,\d/\d \tau^\gb)$ equals to 
\ben 
-\res_{\gl=\infty}k^{-2}\(k\gl^{k-i} + (\nu_1-\nu_0)\gl^{-i} + O(\gl^{-k-1})\)
\(\gl^{-i'} +O(\gl^{-k-1})  \) \frac{d\gl}{\gl}.
\een
The last residue equals $1/k$ if $i+i'=k$ and $0$ otherwise. To compute 
the $\(-\res_{x=0}\)$-term in the residue pairing, we switch to the 
coordinate $y=Qe^{t_N}/x$ and then, in a neighborhood of $y=\infty$, we
change to the coordinate $\gl$ defined by equation \eqref{change_0}. 
An extra caution is required here since the 1-form in the residue
involves partial derivatives in $\tau^\ga$ and $\tau^\gb$ and the coordinate change 
depends on $t.$ Put $\widetilde{f}_t=f_t(Qe^{t_N}/y).$ Then
differentiation by parts yields
\ben
\(\d_{\tau^\ga} f_t\)\, \omega  = \Big(-\frac{\d\widetilde{f}_t}{\d{\tau^\ga}}
\frac{dy}{y} + \frac{\d t_N}{\d \tau^\ga} d\widetilde{f}_t\Big) = 
\Big( \d_{\tau^\ga}\log y + \frac{\d t_N}{\d \tau^\ga}\Big) d\widetilde{f}_t. 
\een  
In the last formula if we change from $y$ to $\gl$ then we get
\ben
\(\d_{\tau^\ga} f_t\)\, \omega = 
\Big(\frac{\d t_N}{\d \tau^\ga} + O(\gl^{-m-1})\Big)
\(m\gl^m + (\nu_0-\nu_1) \) \frac{d\gl}{\gl}. 
\een
From here we find that the $\(-\res_{x=0}\)$-term in the residue pairing 
of $(\d/\d \tau^\ga,\d/\d \tau^\gb)$ equals to
\ben
-\res_{\gl=\infty}
\Big(\frac{\d t_N}{\d \tau^\ga}\frac{\d t_N}{\d \tau^\ga} + O(\gl^{-m-1})\Big)
\(m\gl^m + (\nu_0-\nu_1) \) \frac{d\gl}{\gl}.
\een
On the other hand $t_N = (\tau^{0/k}-\tau^{0/m})/(\nu_0-\nu_1).$ Therefore,
the above residue is 0 unless $\ga=\gb = 0/k,$ in which case it equals 
$\frac{1}{\nu_0-\nu_1}.$
\qed

We trivialize the tangent bundle $TM\iso M\times H$ via the flat coordinates,
i.e., $\d/\d \tau^\ga\mapsto {\bf 1}_\ga.$ Let us denote by $\bullet_\tau'$ 
the multiplication in the tangent space $T_\tau M\iso H.$ 
\begin{conjecture}
The equivariant cup product $\bullet_\tau$ coincides with $\bullet_\tau'.$
\end{conjecture}
This may be interpreted as saying that $f_t$ is the equivariant mirror of $\CC_{k,m}$.

As discussed above, we know that the conjecture holds for $\tau=t_Np$.

Suppose that $k$ and $m$ are co-prime, then the equivariant orbifold quantum cohomology of $\CC_{k,m}$ is multiplicatively generated in degree $2$. Thus the conjecture follows from the reconstruction result of abstract quantum D-module (\cite{Ir}, Theorem 4.9).

For general $k,m$,  we know from our previous article \cite{MT} that in the non-equivariant limit $\nu_0=\nu_1=0$ the conjecture also holds. 
We expect that there should be a reconstruction-type theorem in equivariant quantum 
cohomology that implies the conjecture from the facts that we already know. 
However, reconstruction of non-conformal
(i.e. the structure constants of the cup product are not homogeneous functions) 
Frobenius manifolds, is a topic in Gromov--Witten theory not explored yet.
On the other hand it is not entirely true that the homogeneity is lost, because we
can assign degree 2 (or 1 if we work with complex degrees) to each of the 
characters $\nu_0$ and $\nu_1$ and then the structure constants will be 
homogeneous. However, we could not find a way to use this homogeneity property.

We would like to remark that the Frobenius manifold $M$ in this section
is a slight generalization of the Frobenius structure on the space
of orbits of the extended affine Weyl group of type $A$, introduced by 
B. Dubrovin in \cite{D}. In particular our arguments are parallel to 
the ones in \cite{D}. Apparently, a similar Frobenius manifold was 
introduced by J. Ferguson and I. Strachan (see \cite{FS}) in their study 
of logarithmic deformations of the dispersionless KP-hierarchy. 

\subsection{Oscillating integrals}
Let $\tau\in M\iso H$ be such that $f_\tau$ is a Morse function. Denote by 
$\xi_i\in \C^*,$ $i=1,2,\ldots, k+m$ the critical points of $f_\tau.$
For each $i$ we choose a semi-infinite homology cycle $B_i$ in 
\ben
\lim_{M\rightarrow \infty} 
H_1(Y_\tau, \{ {\rm Re }({f_\tau/z})<-M\};\Z)\iso \Z^{k+m}.
\een 
as follows. Pick a Riemannian metric on $\C^*$ and let $B_i$ be the union of the gradient
trajectories of ${\rm Re} (f_\tau/z)$ which flow into the critical point $\xi_i.$ 
We remark that the function $f_\tau$ is multivalued, however its gradient is a
smooth vector field on $\C^*$ and so the above definition of $B_i$ makes sense. 
Let $\mathfrak{J}:\C^{k+m} \rightarrow H$ be the linear map defined by 
\beq\label{o_int}
(\mathfrak{J}e_i, {\bf 1}_\ga ) := (-2\pi z)^{-1/2}\int_{B_i} e^{f_\tau/z} \phi_\ga \omega,
\eeq
where $e_i,$ $i=1,2,\ldots , k+m$ is the standard basis of $\C^{k+m},$  
the index $\ga \in \Z_k\sqcup\Z_m,$ and we also fixed a choice of a branch 
of $f_\tau$ in a tubular neighborhood of the cycle $B_i.$

Using the method of stationary phase asymptotics (e.g. see \cite{AGV}) we get
that the map $\mathfrak{J}$ admits the following asymptotic: 
\beq\label{J_as}
\mathfrak{J}\sim \Psi (1+R_1 z+ R_2 z^2+\ldots) e^{U/z},\mbox{ as } z\rightarrow 0,
\eeq    
where $R_1, R_2,\ldots $ and $U = \diag (u_1, \ldots ,u_{k+m})$ ($u_i=f_\tau(\xi_i)$ are 
the critical values of $f_\tau$) are linear 
operators in $\C^{k+m},$ and $\Psi:\C^{k+m}\rightarrow H$ is a linear isomorphism
(independent of $z$). Under the isomorphism $\Psi,$ the product $\bullet_\tau$ and
the residue pairing are transformed respectively into
\ben
e_i\bullet_\tau e_j = \Delta_i^{1/2}\delta_{i,j} e_j,\quad (e_i,e_j) = \delta_{i,j},
\een 
where $\delta_{i,j}$ is the Kronecker symbol and $\Delta_i$ is the Hessian 
of $f_\tau$ at the critical point $\xi_i$ with respect to the volume form $\omega_\tau,$
i.e., choose {\em a unimodular} coordinate $t$ in a neighborhood of $\xi_i$ so that
$\omega = dt$ and then $\Delta_i = \d_t^2f_\tau(\xi_i).$  We will write $R=1+R_1z+R_2z^2+...$.

We are ready to define the function $\D^{\rm Fr}.$ However before we
do this let us list two more facts which are not needed in the sequel
but will be important for proving that $\D^{\rm Fr}$ coincides 
with the total descendent potential of $\CC_{k,m}.$ 

\begin{theorem}\label{de_oint}
The map $\mathfrak{J}$ satisfies the following differential equations:
\beq\label{mde}
z\d_{\tau^\ga} \mathfrak{J} = 
\(\phi_\ga \bullet_\tau' \)\, \mathfrak{J},\quad \ga\in \Z_k\sqcup\Z_m.
\eeq
\end{theorem}
The proof of this theorem will be omitted because it is the same as 
the proof of Lemma 3.1 in \cite{MT}.

Assume that $\tau=t_Np$ and that the critical points $\xi_i$ of $f_\tau$ are 
numbered in such a way that 
\ben
\xi_i = \nu^{1/k} + \ldots,\ 1\leq i\leq k \mbox{ and }
\xi_{k+j} = Qe^\tau\overline{\nu}^{-1/m} +\ldots ,\ 1\leq j\leq m,  
\een
where the two groups of expansions are obtained by solving $f_\tau'(x)=0$ 
respectively in a neighborhood of $x=\infty$ and $x=0,$ the dots stand for higher
order terms in $Q,$ and the index $i$ (resp. $j$) corresponds to a 
choice of $k$-th root of $\nu$ (resp. $m$-th root of $\overline{\nu}$).
Put 
\ben
&&
g_{\ga i} := {g_\ga}\nu^{j/k-1/2}, \ \ga=j/k\in\Z_k\  1\leq j\leq k,\ 1\leq i\leq k,\\
\notag
&&
g_{\ga i} := {g_\ga} 
\overline{\nu}^{j/m-1/2}, \ 
\ga=j/m\in\Z_m\  1\leq j\leq m,\ k+1\leq i\leq k+m,
\een
\begin{lemma}\label{asympt_classical}
The asymptotical solution admits a classical limit $Q=0$ which is
characterized as follows: $( \Psi R e_i,{\bf 1}_\ga),$ turns into either
\beq\label{as_k}
g_{\ga i}
\exp \(\sum_{n=2}^{\infty} \frac{B_n(1-j/k)}{n(n-1)}(-\nu)^{-n+1} z^{n-1}\)
\eeq
if $\ga=j/k$, $1\leq j\leq k$, or 
\beq\label{as_m}
g_{\ga i}
\exp \(\sum_{n=2}^{\infty} \frac{B_n(1-j/m)}{n(n-1)}(-\overline{\nu})^{-n+1} z^{n-1}\),
\eeq
if  $\ga = j/m\in\Z_m,$ $1\leq j\leq m,$ where $B_n(x)$ are the Bernoulli polynomials:
\ben
\frac{e^{tx}t}{e^t-1} = \sum_{n=0}^\infty B_n(x)\frac{t^n}{n!}. 
\een
\end{lemma}
\proof
It is enough to verify the first asymptotic, because for the second one we can employ the
symmetry: switch $\nu_0$ with $\nu_1$ and $k$ with $m.$
We have to compute the asymptotic of \eqref{o_int} up to higher order terms in $Q.$ 
Therefore  we can use 
$x^k + (\nu_1-\nu_0)\log x$ instead of $f_\tau$ and also we can assume 
that $e_i$ corresponds to the critical point $\xi_i$, $1\leq i\leq k.$ 
Let us make the substitution 
$t=x^k.$ Then the integral \eqref{o_int}, modulo higher order terms in $Q$, turns into
\beq\label{a_osc}
k^{-1}(-2\pi z)^{-1/2}\int_B e^{(t-\nu\log t)z^{-1}}\phi_\ga(t^{1/k})t^{-1}dt,
\eeq
where the cycle $B$ is constructed via Morse theory for  ${\rm Re}\ (t-\nu\log t)/z$
(see the construction of $B_i$ in \eqref{o_int}). 

More generally, we will compute explicitly the asymptotic as $z\rightarrow 0$ of 
\beq\label{qo_int}
I(\nu,z,s)=\int_B e^{(t-\nu\log t)z^{-1}}t^{s-1}dt,
\eeq
where $s>0$ is any real number. 
Using the method of stationary phase asymptotic (see \cite{AGV}) we get that 
\eqref{qo_int} admits an asymptotic as $z\rightarrow 0$ of the following type: 
\ben
e^{(\nu-\nu\log\nu)/z} \nu^{s-1} (-2\pi \nu z)^{1/2} e^{\sum_{n=2}^\infty A_n(s)(-z/\nu)^{n-1}}.
\een 
In order to verify that the sum in the exponent depends on $z/\nu$ note that the integral 
\eqref{qo_int} satisfies the differential equation $(z\d_z+\nu\d_\nu)I=((-\nu/z)+s)I.$
Furthermore, one checks that $I$ satisfies the differential equation $(z\d_\nu+\d_s)I=0$
which imposes the following recursive relations on the polynomials $A_n$:
\beq\label{a_de}
A_2'(s) = s-\frac{1}{2}, \quad A_{n+1}'(s) = -A_n(s).
\eeq
On the other hand when $s=1$ the asymptotic of \eqref{qo_int} is easily expressed 
in terms of the asymptotic of the Gamma function: 
\ben
(-z)^{-\nu/z+1}\Gamma\(-\frac{\nu}{z}+1\)\sim
e^{(\nu-\nu\log\nu)/z} (-2\pi \nu z)^{1/2} 
e^{\sum_{n=1}^\infty  \frac{B_{2n}}{2n(2n-1)}(-z/\nu)^{2n-1}},
\een
where $B_n=B_n(0)$ are the Bernoulli numbers and the asymptotic of the Gamma function 
is well known (e.g. see \cite{BH}). Thus the coefficient $A_n$ satisfy the 
following initial condition $A_n(1)= B_n/(n(n-1))$ (note that for $n\geq 2$ the odd
Bernoulli numbers vanish), which together with \eqref{a_de} uniquely determines $A_n.$
Using that the Bernoulli polynomials satisfy the identity: $B_n'(x) = nB_{n-1}(x)$, it is
easy to verify that $A_n(s) = B_n(1-s)/(n(n-1)).$ 
\qed
\begin{remark}
Lemma \ref{asympt_classical} implies Givental's {\em R-conjecture} for $\CC_{k,m}$.
\end{remark}

\subsection{The symplectic loop space formalism}
Let $\H:=H((z^{-1}))$ be the space of formal Laurent series in $z^{-1}$ 
with coefficients in $H.$ We equip $\H$ with the symplectic form:
\ben
\Omega(\f(z),\g(z)):= \res_{z=0} \(\f(-z),\g(z)\) dz. 
\een
Let $\{ {\bf 1}^\ga\}_{\ga\in \Z_k\sqcup\Z_m}$ be a basis of $H$ 
dual to $\{ {\bf 1}_\ga\}$ with respect to the {\Poincare} pairing. 
Then the functions 
$p_{n,\ga} = \Omega(\ , {\bf 1}_\ga z^n)$ and 
$q_n^\ga = \Omega({\bf 1}^\ga (-z)^{-n-1},\ ),$ where $n\geq 0$ 
and $\ga\in \Z_k\sqcup\Z_m $ form a Darboux coordinate system on $\H.$ 
We quantize functions on $\H$ via the Weyl's quantization rules: 
the coordinate functions ${p}_{n,\ga}$ and $q_n^\ga$ are 
represented respectively by the differential operator 
$\widehat{p}_{n,\ga} = \ge\,\d/\d q_n^\ga$ and the multiplication 
operator $\widehat{q}_{n}^\ga = \ge^{-1}\,q_n^\ga,$ and we demand  
normal ordering, i.e., always put the differentiation
before the multiplication operators.

If $A$ is an infinitesimal symplectic transformation of $\H$ then the map
$\f \mapsto A \f$ determines a linear {\em Hamiltonian} vector field. 
It is straightforward to verify that the corresponding Hamiltonian
coincides with the quadratic function $h_A:=-\frac{1}{2}\Omega(A\f,\f).$
By definition $\widehat{A}:= \widehat{h}_A.$ If $M$ is a symplectic 
transformation of $\H$ such that $A:=\log M$ exist then we define 
$\widehat{M} := e^{\widehat{A}}.$

From now on we will consider only $\tau=t_N p.$ Put
\beq\label{Dsc}
\D^{\rm Fr}=
C(\tau)\widehat{S}_\tau^{-1}\(\Psi Re^{U/z}\)\sphat \ \prod_{i=1}^{k+m} \D_{\rm pt}(\q^i),
\eeq
where the vector space $H$ is identified with the standard vector space $\C^{k+m}$
via $\Psi$ and $\q^i$ are the coordinates of $\q\in H[z]$ with respect to the standard
basis, i.e., $\sum \q^i(z)e_i = \Psi^{-1}\q(z),$ and $\D_{\rm pt}$ is the total descendent
potential of a point:
\ben
\D_{\rm pt}(\t) = \exp \(\sum_{n,g} \ge^{2g-2}\frac{1}{n!} \int_{\overline{M}_{g,n}}
\prod_{j=1}^n (\t(\psi_j)+\psi_j)\), 
\een 
where $\t(z)=t_0+t_1 z+\ldots\in \C[z].$ The factor $C$ in \eqref{Dsc} is a complex-valued 
function on $H$ such that it makes the RHS independent of $\tau.$ For all 
further purposes $C(\tau)$ is irrelevant and it will be ignored.

\sectionnew{Vertex operators and the equivariant mirror model of $\CC_{k,m}$}

\subsection{Introduction}

Given a vector $\f\in \H$, the corresponding linear function 
$\Omega(\ ,\f)$ is a linear combination of $p_{n,\ga}$ and $q_n^\ga$
and $\widehat \f$ is defined by the above rules. Expressions like
$e^\f,\ \f\in\H$ are quantized by first decomposing $\f=\f_-+\f_+$, where 
$\f_+$ (respectively $\f_-$) is the projection of $\f$ on $\H_+ :=H[z]$ 
(respectively $\H_-:=z^{-1}H[[z^{-1}]]$), and then setting
$\(e^ \f\)\sphat = e^{\widehat \f_-}e^{\widehat \f_+}$. Note that
the vertex operators in the introduction are quantized exactly
in this way.

The proof of \thref{t1} amounts to conjugating the vertex operators $\Gamma^\pm$ and 
$\overline{\Gamma}^\pm$ by the symplectic transformation $\widehat{S}_\tau$ and then 
by $\Psi R e^{U/z}.$  For the first conjugation we use the following formula
(\cite{G1}, formula (17)):
\beq
\label{s_conj}
\widehat S_\tau e^{\hat \f}\widehat S_\tau^{-1}  = 
e^{W(   \f_+  ,  \f_+   )/2}
e^{      (S_\tau \f) \sphat                      },
\eeq 
where $\f\in \H$ and $+$ means truncating the terms corresponding to the 
negative powers of $z$ and the quadratic form 
$W(\f_+,\f_+)=\sum (W_{nl}\f_l,\f_n)$ is defined by
\beq
\label{conj_s}
W_{nl}w^{-n}z^{-l} = \frac{S_\tau^*(w)S_\tau(z)-1}{w^{-1}+z^{-1}}.
\eeq 
Therefore, our next goal is to compute $S_\tau\f^\pm$ and 
$S_\tau\overline{\f}^\pm.$ Before doing so we explain a very important property of our vertex operators. The
content of the next section is the key to the proof of \thref{t1}.

\subsection{Changing the coordinate $\gl$} 
Let us denote by $\O$ the space of formal Laurent series in $\gl^{-1}$
and by $\O[[z^{\pm1}]]$ the space of formal series:
\beq\label{element_o}
\f(\gl,z) = \sum_{n\in \Z} \I^{(n)}(\gl)(-z)^n,
\quad \mbox{ such that } \quad
\lim_{n\rightarrow\infty} \I^{(n)}(\gl) = 0,
\eeq
where the limit is understood in the $\gl$-adic sense, i.e., 
for each $N>0$ there exist $d\in \Z$ such that 
$\I^{(n)}\in \gl^{-N}\C[[\gl^{-1}]]$ for all $n\geq d.$

Furthermore, we fix an element $\phi \in \O$ such that both
$\res_{\gl=\infty}\phi$ and the polynomial part $p\in \C[\gl]$ of
$\phi$ are non-zero and we introduce the following first order 
differential operator:
\beq\label{dop}
D = -zp^{-1}\d_{\gl} - p^{-1}\phi_-,
\eeq
where $\phi_-:= \phi - p.$

Let $\g\in \O[[z^{\pm 1}]]$ be a Laurent series in $\gl^{-1}$ and $z^{-1}$:
\beq\label{g}
\g = \sum_{\substack { a\geq A \\ r\geq R}} g_{a,r} \gl^{-a} z^{-r}.
\eeq
We will prove that the operator $D$ is a linear isomorphism in $\O[[z^{\pm 1}]]$
and that the infinite sum
\beq\label{Dg}
\f = \sum_{n\in \Z} \I^{(n)}(\gl)\,(-z)^n:= \sum _{n\in \Z} D^n \g,
\eeq
is a well defined element of $\O[[z^{\pm 1}]]. $
The main result in this subsection is the following transformation law 
for $\f.$
\begin{proposition}\label{transformation}
If $x=\gl + a_0 + a_1\gl^{-1}+\ldots $ is 
another formal coordinate near $\gl=\infty.$ Then
\ben
\f(x) = \f(\gl)\exp\(z^{-1}\int^\gl_x \phi(t) dt \) \ .
\een
\end{proposition}
\proof 
Note that $D\f=\f,$ i.e., 
\ben
\d_\gl \I^{(n)}(\gl) = \phi(\gl)\, \I^{(n+1)}(\gl).
\een 
Thus the same proof as in \cite{M1}, Lemma 3.2, applies. \qed

We will show that for each pair of positive integers $M$ and $N$ there exists 
$d\in \Z$ such that
\beq\label{zl_adic}
 D^{-M'}\g\, \in\, z^{-M}\O[[z^{-1}]]\quad \mbox{ and }\quad
 D^{N'} \g\, \in\, \gl^{-N} \C[[\gl^{-1},z^{\pm 1}]]
\eeq
for all $M'>d$ and $N'>d.$ This would imply that 
the infinite sum \eqref{Dg} is convergent in an appropriate $z,\gl$-adic
sense to some element in $\O[[z^{\pm 1}]].$ 

We pass to a new variable $\xi = \int p(\gl) d\gl.$ If $k-1$ is the degree of
the polynomial $p$ then, after inverting the change, we see that 
$\O\iso \C((\xi^{-1/k})).$ Also the differential operator $D$ takes the form
\ben
D_\xi = -z\d_\xi + \nu/\xi  + \sum_{i\geq 1} a_i \xi^{-1-i/k},
\een
where $a_i$ are some constants and $\nu\neq 0.$
   
\begin{lemma}\label{D}
The operator $D_\xi$ is a linear isomorphism in $\O[[z^{\pm 1}]]$ and its inverse
has the following property:  
\ben
D_\xi^{-1} \xi^\ga \ \in\  
\begin{cases}
z^{-1}\O[[z^{-1}]] & \mbox{ if $\ga \neq -1$}, \\
\nu^{-1} + z^{-1}\O[[z^{-1}]] & \mbox{ otherwise},
\end{cases}
\een
where $\ga \in (1/k)\Z.$
\end{lemma}
\proof
We will construct the inverse of $D_\xi.$ The equation $D_\xi\f = \xi^\ga$
has a unique solution of the following form
\ben
\f^\ga  =  \frac{1}{-z(\ga+1)+ \nu}\, \xi^{\ga+1} + 
\sum_{j\geq 1}f^\ga_j \xi^{\ga+1-j/k}, 
\een 
where $f^\ga_j\in z^{-1}\C[[z^{-1}]].$   
We define $D_\xi^{-1}\xi^\ga:=\f^\ga$ and one checks 
that if $D_\xi^{-1}$ is extended by linearity, then 
$D_\xi^{-1}\f\in \O[[z^{\pm 1}]]$ for all $\f\in\O[[z^{\pm 1}]].$ The lemma follows. 
\qed

Assume that $\g$ is a series of the type \eqref{g}, i.e., the powers of $z$ and 
$\gl$ are bounded from above. According to \leref{D}, the operator 
$D^{-2}$ will decrease the highest degree of $z$ at least by $1$. On the other
hand note that $D$ decreases the highest degree of $\gl$ at least by $k$. 
Thus \eqref{zl_adic} holds.

\subsection{The symplectic action on $\f^\pm$}{\label{Sf}}

Let  $D_x$ be the differential operator \eqref{dop} corresponding
to $\phi=\d_x f_\tau$, i.e., 
\beq\label{D_x}
D_x = -z \frac{1}{k} x^{1-k} \d_x + \frac{1}{k}(\nu_0-\nu_1) x^{-k} + 
\frac{m}{k}(Qe^\tau)^m x^{-k-m}.
\eeq
We define a vector in the symplectic loop space $\H$ 
\ben
\f^\pm_\tau = \sum_{n\in\Z} I^{(n)}_\pm(\tau,x)(-z)^n,
\mbox{ s.t. }
(\f^\pm_\tau,{\bf 1}_\ga):= \pm \sum_{n\in\Z} D_x^n \(k^{-1}\phi_\ga(x) x^{-k}\),
\een
where $\ga \in \Z_k\sqcup\Z_m.$ 
Let us compute $I^{(0)}_\pm(\tau,x).$ Note that the terms in the above sum which
contribute to $I^{(0)}_\pm$ are the ones with $n\geq 0.$  The rest, according to
\leref{D}, do not contribute. Thus 
\beq\label{I0}
(I^{(0)}_\pm(\tau,x),{\bf 1}_\ga) =  \phi_\ga(x) \frac{\omega}{d f_\tau}.
\eeq 
Note that $D_x\f_\tau^\pm=\f_\tau^\pm $, thus by comparing the coefficients in front 
of $(-z)^{n+1}$ we get the following recursive relation:
\beq\label{In}
\d_x I^{(n)}(\tau,x) = (\d_x f_\tau)\,I^{(n+1)}(\tau,x).
\eeq
In particular, all coefficients $I^{(n)}$ are rational vector-valued functions
on $Y_\tau$ with possible poles only at the critical points of $f_\tau.$

In a neighborhood of $x=\infty$ we choose another (formal)
coordinate $\gl = x+ a_0 + a_1x^{-1}+\ldots$ such that $\gl$
is a formal solution to the equation
\beq\label{change}
\gl^k + \nu_1 \log \gl + \nu_0\log(Q/\gl) = f_\tau(x),
\eeq
where $f_\tau(x) = f(x,Qe^\tau/x).$

We will show that $S_\tau \f^\pm(\gl) = \f_\tau^\pm(x).$ It is enough to prove that 
\ben
\sum_{n\in\Z} D_x^n \(k^{-1}\phi_\ga(x)x^{-k}\) = 
\(\f_\tau^\pm(x),{\bf 1}_\ga \) = \(S_\tau \f^\pm(\gl),{\bf 1}_\ga \)=
\(\d_\ga J, \f^\pm(\gl)\),
\een
where $\ga \in \Z_k\sqcup\Z_m,$ $\d_\ga$ is the derivative along vector ${\bf 1}_\ga,$
and the last equality is deduced after comparing the definitions of the $J$-function $J$ and $S_\tau$ of $\CC_{k,m}$.

Let us compute $(\d_\ga J,\f^\pm).$
Assume first that $\ga = i'/k, 1\leq i'\leq k.$ Note that only the first sum 
in the formula for $z\d_\ga J$ (see \coref{dJ}) will contribute to the inner product.
Take the $d$-th summand in this sum.  It will
have a non-zero pairing only with those terms in $\f^\pm(\gl)$, (see \eqref{vop_0})
which correspond to $n\in \Z$ and $i,1\leq i\leq k$ s.t. 
\ben
-dm+i' + k-i = 0\,({\rm mod }\, k),\quad {\rm i.e.},\quad i= -dm+i'\,({\rm mod }\, k). 
\een
Pick  $n\in \Z$ in such a way that the product in $\f^\pm$ corresponding
to $n$ and $i$ cancels with the product in the $d$-th summand, i.e., 
$ -i/k + n = (dm-i')/k.$ On the other hand, note that the sum of all 
terms in $\f^\pm$ which have a non-zero pairing with the $d$-th summand 
can be written as follows: 
\ben
\sum_{n'\in\Z} \widetilde{D}_\gl^{n'} 
\frac{ \prod_{l=-\infty}^n (\nu + (-i/k+l)z)}
     { \prod_{l=-\infty}^0 (\nu + (-i/k+l)z)} \ 
\gl^{-(n+1)k+i} \ {\bf 1}_{(k-i)/k} ,
\een
where
\ben
\widetilde{D}_\gl = -z\frac{1}{k}\gl^{1-k}\d_\gl +\frac{1}{k}(\nu_0-\nu_1)\gl^{-k} .
\een
Thus the pairing between $z\d_\ga J$ and $\f^\pm$ is 
\ben
g_\ga e^{\tau\nu_0/z}\sum_{n'\in \Z} D_\gl^{n'} \sum_{d\geq 0} \gl^{-k+i'} 
\frac{1}{d!}\[z^{-1}\(Qe^{\tau}/\gl\)^m\]^d =
e^{\[\(Qe^{\tau}/\gl\)^m+\tau\nu_0\]z^{-1}} \sum_{n'\in \Z} D_\gl^{n'} g_\ga\gl^{-k+i'},
\een
where $D_\gl$ is given by formula \eqref{D_x}. We recall \prref{transformation}, the
change \eqref{change} and since $g_\ga\gl^{i'} = k^{-1}\phi_\ga(\gl)$ we get exactly 
what we wanted to prove. 
The case when $\ga=j/m,1\leq j\leq m$ is similar and will be omitted.

A similar statement holds for the other vertex operators $\overline{\Gamma}^\pm.$
Let $y=Qe^{t_N}/x$ be another coordinate on the complex circle.
Put $f_\tau(y) = f(Qe^\tau/y,y)$ and let
$D_y$ be the differential operator \eqref{dop} corresponding
to $\phi=\d_y f_\tau$, i.e., 
\beq\label{D_y}
D_y = -z \frac{1}{m} y^{1-m} \d_y + \frac{1}{m}(\nu_1-\nu_0) y^{-m} + 
\frac{k}{m}(Qe^\tau)^k y^{-k-m}.
\eeq
We define a vector in the symplectic loop space $\H$ 
\ben
\overline{\f}^\pm_\tau = \sum_{n\in\Z} \overline{I}^{(n)}_\pm(\tau,y)(-z)^n,
\mbox{ s.t. }
(\overline{\f}^\pm_\tau,{\bf 1}_\ga):= 
\pm \sum_{n\in\Z} D_y^n \(m^{-1}\phi_\ga(y) y^{-m}\).
\een
Just like before we prove that the 0-mode is given by
\beq\label{bar_I0}
(\overline{I}^{(0)}_\pm(\tau,y),{\bf 1}_\ga) =  \phi_\ga(y) \frac{\omega}{d f_\tau}.
\eeq 
and that the following recursive relation holds:
\beq\label{bar_In}
\d_y \overline{I}^{(n)}(\tau,y) = (\d_y f_\tau)\,\overline{I}^{(n+1)}(\tau,y).
\eeq
In a neighborhood of $y=\infty$ we choose a formal coordinate $\gl$ such that
\beq\label{change_y}
\gl^m +  \nu_0 \log \gl + \nu_1\log(Q/\gl) = f_\tau(y).
\eeq
Then $S_\tau\overline{\f}^{\,\pm}(\gl) = \overline{\f}_\tau^{\,\pm}(y).$


\sectionnew{From descendants to ancestors \label{d_to_a}  }


Let us describe the HQE which one obtains after conjugating the HQE in 
\thref{t1} by $S_\tau$ and then we will give the details of the computation.

An {\em asymptotical function} is, by definition, an expression 
\ben
\T = \exp \(\sum_{g=0}^\infty \ge^{2g-2}\T^{(g)}(\t;Q)\),
\een
where $\T^{(g)}$ are formal series in the sequence of vector 
variables $t_0,t_1,t_2,\ldots$  with coefficients in the Novikov
ring $\C[[Q]].$ Furthermore, $\T$ is called {\em tame} if
\ben
\left.\frac{\d}{
\d t_{k_1}^{\ga_1} \ldots \d t_{k_r}^{\ga_r} }\right|_{\t=0} 
\T^{(g)} = 0 \quad \mbox{whenever} \quad
k_1+k_2+\ldots +k_r > 3g-3+r,
\een
where $t_{k}^{\ga}$ are the coordinates of $t_k$ with respect to 
$\{{\bf 1}_\ga\}.$
We will say that a tame asymptotical function $\T$ satisfies the HQE below if for each
integer $r$ 
\beq\label{hqe_a}
\(\res_{x=0} + \res_{x=\infty}\)\ c_r(\tau,x)\, \(\Gamma_\tau^-\tensor\Gamma_\tau^+\)
\(\T\tensor \T\) dx = 0,
\eeq
where $\Gamma_\tau^\pm$ are the vertex operators $e^{\widehat{\f}_\tau^\pm}$ (see
subsection \ref{Sf}) and 
\ben
c_r(\tau,x)=x^{-r-1}\exp\(
(r-1)\, \frac{x^k}{\nu_0-\nu_1} + (r+1)\, \frac{(Qe^\tau/x)^m}{\nu_0-\nu_1} \).
\een
The Hirota quadratic equations \eqref{hqe_a} are interpreted as follows: 
switch to new variables $\x$ and $\y$ via the substitutions: 
$\q'=\x+\ge\y$, $\q''=\x-\ge\y$.
Due to the tameness 
(\cite{G1}, section 8, Proposition 6), after canceling the terms 
independent of $x$, the 1-form on the LHS of \eqref{hqe_a} 
expands into a power series in $\y$ and $\ge$, such that each coefficient 
depends polynomially on finitely many $I^{(n)}(\tau,x)$ and finitely 
many partial derivatives of $\T$. The residues in 
\eqref{hqe_a} are interpreted as the residues of 
meromorphic 1-forms.  

According to A. Givental (\cite{G2}, section 8), the asymptotical function 
$\A_\tau^{\rm Fr}:=\(\Psi R e^{U/z}\)\sphat \prod \D_{\rm pt}(\q^i)$ is tame.
Slightly abusing the notations we use $\tau\in \C$ to denote also the 
cohomology class $\tau\,p.$ The goal in this section is to prove the following theorem.
\begin{theorem}\label{t3}
$\D^{\rm Fr}$ satisfies \eqref{HQE} iff $\A_\tau^{\rm Fr}$ satisfies \eqref{hqe_a}. 
\end{theorem}
\proof
We recall formula \eqref{s_conj} and the main result in subsection \ref{Sf}: 
\ben
\widehat S_\tau \Gamma^\pm\widehat S_\tau^{-1}  = 
e^{W/2}\Gamma_\tau^\pm,\quad
\widehat S_\tau \overline{\Gamma}^\pm\widehat S_\tau^{-1}  = 
e^{\overline{W}/2}\overline{\Gamma}_\tau^\pm,
\een
where 
\ben
W = W_\tau((\f^+(\gl))_+,(\f^+(\gl))_+),\quad
\overline{W} = W_\tau((\overline{\f}^+(\gl))_+,(\overline{\f}^+(\gl))_+).
\een
We will prove that 
\beqa\label{w}&&
W = C + \frac{2}{\nu_0-\nu_1} \(Qe^\tau/x\)^m + 
\log\frac{\gl(k\gl^{k} +\nu_1-\nu_0)}{x^2\d_xf_\tau},  \\
&& \label{w_bar}
\overline{W} 
= \overline{C}+ \frac{2}{\nu_1-\nu_0} \(Qe^\tau/y\)^k + 
\log\frac{\gl(m\gl^{m} +\nu_0-\nu_1)}{y^2\d_yf_\tau}  \,
\eeqa  
where 
$C=\(S_1{\bf 1}_{0/k},{\bf 1}_{0/k}\),$ 
$\overline{C}=\(S_1 {\bf 1}_{0/m},{\bf 1}_{0/m}\),$ and $x$ and $y$ are
related to $\gl$ respectively via \eqref{change} and \eqref{change_y}. 
It is sufficient to establish the first formula, because for the second one 
one just interchange $k$ with $m,$ $\nu_0$ with $\nu_1,$ and $x$ with $y.$ 
Using that $\d_x I_+^{(k)} = (\d_x f_\tau)I_+^{(k+1)}$ we get
{\allowdisplaybreaks
\ben
\d_xW & = &
\d_x W_\tau((\f^+)_+,(\f^+)_+) = 
\d_x\ \sum_{n,l\geq 0} (W_{nl}I_+^{(l)},I_+^{(n)})(-1)^{n+l} = \\
& = &
-\sum_{n,l\geq 0} \[([W_{n,l-1}+W_{n-1,l}]I_+^{(l)},I_+^{(n)}) \]
\(\d_xf_\tau\)  (-1)^{n+l}= \\
& = &
-\sum_{n,l\geq 0} \[ (S_l(-1)^lI_+^{(l)},S_n(-1)^n I_+^{(n)}) -
                    ( I_+^{(0)}(\gl)   ,I_+^{(0)}(\gl)) \]
\(\d_xf_\tau \)= \\
& = & 
\[ -\(        I_+^{(0)}(\tau,x),I_+^{(0)}(\tau,x)    \)+ 
    \(        I_+^{(0)}(\gl)   ,I_+^{(0)}(\gl)       \)    \]
    \(\d_xf_\tau \),
\een }
where $\gl$ and $x$ are related via equation \eqref{change}. The two 1-forms $\(I_+^{(0)}(\tau,x),I_+^{(0)}(\tau,x) \)df_\tau$ and $\(I_+^{(0)}(\gl)   ,I_+^{(0)}(\gl)\)df_\tau$ are equal respectively to 
\ben
\(k\nu^{-1}x^{2k}+m\overline{\nu}^{-1}(Qe^\tau/x)^{2m} +k(k-1)x^k+m(m-1)(Qe^\tau/x)^{m}\)
\frac{dx}{x^2f_\tau'}
\een
and
\ben
\((k-1)\gl^k + \nu^{-1}\gl^{2k}\)\frac{d\gl}{\gl(\gl^k-\nu)}.
\een
One can check that primitives of these two 1-forms are given respectively by
\beq\label{pr}
\log\(x^2f'_\tau\) + \frac{f_\tau(x)-2\(Qe^\tau/x\)^m}{\nu_0-\nu_1}\,
\mbox{ and }
\frac{\gl^{k}}{\nu_0-\nu_1}+\log \(\gl^k-\nu \).
\eeq
In order to fix the integration constant $C$, note that $(\f^+)_+ ={\bf 1}_{0/k}$ for
$x=\infty.$ Thus 
$$
C = W|_{x=\infty} = W_\tau({\bf 1}_{0/k} ,{\bf 1}_{0/k} ) = 
\(W_{0,0}{\bf 1}_{0/k},{\bf 1}_{0/k}\) = \(S_1{\bf 1}_{0/k},{\bf 1}_{0/k}\).
$$

The rest of the proof follows the argument in section 3.5 in \cite{M1}. 
Since $\D=\widehat{S_\tau}^{-1}\A_\tau$ up to a prefactor, we get that $\D$ satisfies
\eqref{HQE} iff $\A_\tau$ satisfies the following HQE:
\beqa\label{hqe_at}
&& \res_{\gl=\infty} \ 
\frac{d\gl}{\gl}
\left(
\gl^{n-l}e^{W}\Gamma_\tau^{-}\tensor\Gamma^{+}_\tau -
(Q/\gl)^{n-l}e^{\bW}\overline{\Gamma}_\tau^{+}\tensor\overline{\Gamma}_\tau^{-}
\right) \\ \notag
&&
\(e^{(n+1)\hat\phi_{0/k}(\tau,z)+n\hat\phi_{0/m}(\tau,z)}\tensor
  e^{l\hat\phi_{0/k}(\tau,z)+(l+1)\hat\phi_{0/m}(\tau,z)} \) 
\(\A_\tau\tensor\A_\tau\) = 0,
\eeqa
where $\phi_{0/k}=S_\tau{\bf 1}_{0/k}$ and $\phi_{0/m}=S_\tau{\bf 1}_{0/m}.$
This is the place where we will use that $\A_\tau$ is a tame asymptotical function.
The tameness condition implies that after the substitutions 
$
\ge\y = (\q'-\q'')/2$ and ${\x} = (\q'+\q'')/2,
$
and the cancellation of terms that do not depend on $\gl$, the 1-form
in \eqref{hqe_at} becomes a formal series in $\y$ and $\ge$
with coefficients depending polynomially on finitely many of the modes 
$I_{\pm}^{(n)}$ and $\overline{I}_\pm^{(n)}$ and finitely many
partial derivatives of $\overline\F_\tau(\x):=\log \A_\tau.$ 
Furthermore, if we choose two new (formal) coordinates 
$x$ and $y$ in a neighborhood of $\gl=\infty$ according to \eqref{change}
and \eqref{change_y} then the coefficients 
$I_\pm^{(n)}$ and $\overline{I}_\pm^{(n)}$ become rational functions 
respectively in $x$ and $y$. 
Thus the LHS of \eqref{hqe_at} is a formal series in $\x,\y$, and $\ge$ 
whose coefficients are residues of rational 1-forms. 
In particular, the action of the translation operator 
\ben
e^{-(n+1)\hat\phi_{0/k}(\tau,z)-n\hat\phi_{0/m}(\tau,z) }\tensor
e^{-l\hat\phi_{0/k}(\tau,z)-(l+1)\hat\phi_{0/m}(\tau,z) }
\een
on \eqref{hqe_at} is well defined, i.e., we can cancel the 
corresponding term in \eqref{hqe_at}. However, since
$e^{\hat f}e^{\hat g} = e^{\Omega(f,g)}e^{\hat g}e^{\hat f},$
the two vertex-operator terms in \eqref{hqe_at} will
gain the following commutation factors: 
\ben
e^{\Omega\( -(n+1)\phi_{0/k}-n\phi_{0/m},\f_\tau^{-}\) + 
\Omega\( -l\phi_{0/k}-(l+1)\phi_{0/m},\f_\tau^{+}\)} = 
e^{(l-n-1)\frac{\gl^k}{\nu_0-\nu_1} }
\een  
and
\ben
e^{
\Omega\(-(n+1)\phi_{0/k}- n\phi_{0/m},\overline{\f}_\tau^+ \)+  
\Omega\(-l\phi_{0/k}-(l+1)\phi_{0/m},\overline{\f}_\tau^- \)}=
e^{(n-l-1)\frac{\gl^m}{\nu_1-\nu_0}}, 
\een  
where we used that $S_\tau$ is a symplectic transformation, thus 
$$\Omega(S_\tau{\bf 1}_\ga,\f_\tau^\pm )=
 \Omega(S_\tau{\bf 1}_\ga,S_\tau\f^\pm )=\Omega({\bf 1}_\ga,\f^\pm ),
\quad \ga\in\Z_k\sqcup\Z_m,$$
and the later is easy to compute from formula \eqref{vop_0}.
Thus \eqref{hqe_at} is equivalent to the following HQE:
\ben
&&
\res_{\gl=\infty} 
\frac{d\gl}{\gl}\times \\
&&
\left(
\gl^{-r} e^{ W +(r-1)\frac{\gl^k}{\nu_0-\nu_1}  }
\Gamma_\tau^-\tensor\Gamma_\tau^+ - 
\(\frac{Q}{\gl}\)^{-r}e^{\bW-(r+1)\frac{\gl^m}{\nu_1-\nu_0} }
\overline{\Gamma}_\tau^+\tensor \overline{\Gamma}_\tau^{-}
\right) 
 \(\A_\tau\tensor\A_\tau \)= 0,
\een
where we put $r=l-n.$ 
We write the above residue sum  as a difference of two 
residues. In the first one we change from $\gl$ to $x$ according to \eqref{change}
and we recall formula \eqref{w}. After a short computation we get:
\beqa
&&\label{x_change} 
\res_{x=\infty}e^{C+(r-1)\tau\frac{\nu_0}{\nu_0-\nu_1}} c_r(\tau,x)
\(\Gamma_\tau^-\tensor\Gamma_\tau^+\)
\(\A_\tau\tensor\A_\tau \) d x ,
\eeqa
where $c_r(\tau,x)$ is the same as in \eqref{hqe_a}. In the second residue
we change $\gl$ to $y$ according to \eqref{change_y} and we recall
formula \eqref{w_bar}:
\ben
\res_{y=\infty}
e^{\overline{C}+(r+1)\tau\frac{\nu_0}{\nu_0-\nu_1}}c_r(\tau,Qe^\tau/y) 
\(\overline{\Gamma}_\tau^+\tensor\overline{\Gamma}_\tau^- \)
\(\A_\tau\tensor\A_\tau \)
Q\frac{dy}{y^2} .
\een
Note that if we change $y=Qe^\tau/x$ then 
$\overline{\Gamma}_\tau^\pm = \Gamma_\tau^\mp,$ thus the last residue transforms into
\beq\label{y_change}
-\res_{x=0}
e^{\overline{C}+(r+1)\tau\frac{\nu_0}{\nu_0-\nu_1}-\tau}c_r(\tau,x) 
\(\Gamma_\tau^-\tensor\Gamma_\tau^+ \)
\(\A_\tau\tensor\A_\tau \)
dx .
\eeq
We compare \eqref{x_change} and \eqref{y_change} and we see that in 
order to finish the proof of the theorem, we just need to verify that 
\ben
C+(r-1)\frac{\tau\nu_0}{\nu_0-\nu_1}=
\overline{C}+(r+1)\frac{\tau\nu_0}{\nu_0-\nu_1}-\tau,\mbox{ i.e. } 
C-\overline{C} = \tau\,\frac{\nu_0+\nu_1}{\nu_0-\nu_1}.
\een
On the other hand we know that $C=(S_1{\bf 1}_{0/k},{\bf 1}_{0/k})$ which is
equal to the coefficient in front of $z^{-1}$ in $(\d_{0/k}J,{\bf 1}_{0/k}).$ 
The later can be computed from \coref{dJ}. The answer is the following:
\ben
C =
\begin{cases}
\tau\nu_0/(\nu_0-\nu_1) & \mbox{ if } k\neq m,\\
\tau\nu_0/(\nu_0-\nu_1)+ k(Qe^\tau)^k/(\nu_0-\nu_1)^2 & \mbox{ if } k=m.
\end{cases}
\een 
Similarly,
\ben
\overline{C} =
\begin{cases}
\tau\nu_1/(\nu_1-\nu_0) & \mbox{ if } k\neq m,\\
\tau\nu_1/(\nu_1-\nu_0)+ m(Qe^\tau)^m/(\nu_1-\nu_0)^2 & \mbox{ if } k=m.
\end{cases}
\een
The theorem follows. \qed


\sectionnew{From ancestors to KdV \label{a_to_kdv}}


In this section we prove that the ancestor potential $\A_\tau^{\rm Fr}$ satisfies
\eqref{hqe_a}. In view of \thref{t3} this would imply \thref{t1}. Note that the vertex
operators $\Gamma_\tau^\pm$ have poles only at $x=0,\infty,$ or $\xi_i,$ $1\leq i\leq k+m$,
where the later are the critical points of $f_\tau.$ Thus it is enough to prove that
the residue of the 1-form in \eqref{hqe_a} at each critical point $\xi_i$ is $0.$ 

Let us fix a critical point $\xi_i$ and denote by $u_i=f_\tau(\xi_i)$ the corresponding 
critical value. The function $f_\tau$ induces a map between a neighborhood of $x=\xi_i$
and a neighborhood of $\Lambda=u_i$ which is a double covering branched at $u_i.$ 
We pick a reference point $\Lambda_0$ in a neighborhood of $u_i$ and denote by 
$x_\pm(\Lambda_0)$ the two points which cover $\Lambda_0.$ Finally, let us denote 
by $x_\pm(\Lambda)$ the points covering $\Lambda$. Note that $x_\pm(\Lambda)$ depend 
on a choice of a path $C$ between $\Lambda_0$ and $\Lambda$ avoiding $u_i.$ 
On the other hand, for any function $g(x)$ meromorphic in a neighborhood of $\xi_i$
we have 
\ben
\res_{x=\xi_i} g(x)dx= \res_{\Lambda=u^i} 
\sum_{\pm}g(x_\pm(\Lambda))\frac{\d x_\pm}{\d\Lambda}(\Lambda)d\Lambda.
\een
Thus the vanishing of the residue at $\xi_i$ of \eqref{hqe_a} is equivalent to:
\beq\label{res_xi}
\res_{\Lambda=u_i}
\left\{   d\Lambda \sum_{\pm} 
\frac{e^{\frac{2}{\nu_0-\nu_1}(Qe^\tau/x_\pm)^m}  }
     {x_\pm^2f_\tau'(x_\pm)}
(\Gamma_\tau^{\gb_\pm}\tensor\Gamma_\tau^{-\gb_\pm})(\A_\tau\tensor\A_\tau)
\right\} 
e^{\frac{r-1}{\nu_0-\nu_1}\Lambda} = 0 ,
\eeq
where $\gb_\pm$ are the one point cycles $[x_\pm(\Lambda)]\in  H^0(f_\tau^{-1}(\Lambda);\Z)$ and
the vertex operators can be described as follows: 
\ben
\f_\tau^{\pm\gb}(\Lambda) = - \int_\gb \f_\tau^\pm(x),\quad 
\Gamma_\tau^{\pm\gb}= \(e^{\f_\tau^{\pm\gb} }\)\sphat, \quad 
\gb\in H^0(f_\tau^{-1}(\Lambda);\Z).
\een
We will prove that the 1-form in the $\{ \ \}$-brackets in \eqref{res_xi} is analytic in 
$\Lambda$. In particular this would imply that the residue \eqref{res_xi} is 0. 
The proof follows closely the argument in \cite{G2}. 

Note that the vector-valued function $I_{\gb_+-\gb_-}^{(0)}(\tau,\Lambda)$ can be expanded
in a neighborhood of $\Lambda=u_i$ as follows
\beq\label{I_0}
I_{(\gb_+-\gb_-)/2}^{(0)}(\tau,\Lambda) = 
\frac{1}{\sqrt{2(\Lambda-u^i)}}\, (e_i + O(\Lambda-u^i)),
\eeq
where the standard vector $e_i$ in $\C^{k+m}$ is identified via $\Psi$ with
a vector in $H$ and the value of $\sqrt{2(\Lambda-u_i)}$
is fixed as follows. Choose a path $C_0$ from $u_i+1$ to $\Lambda_0$, then the
translation of $C\circ C_0$ along vector $-u_i$ is a path from 1 to $\Lambda - u_i.$
If we choose $C_0$ arbitrary then \eqref{I_0} is correct up to a sign, so if necessary
change $C_0$ in order to achieve equality. 
\begin{figure}
\centering
\includegraphics{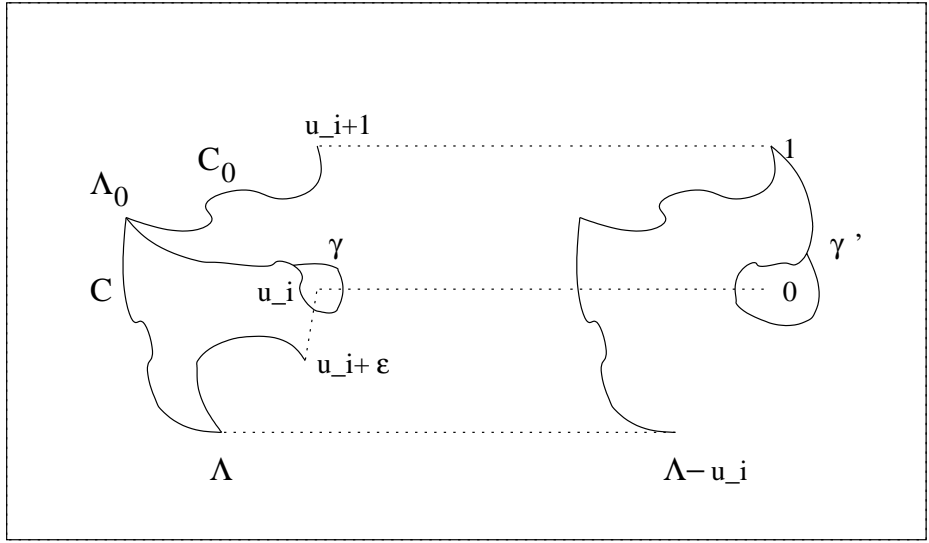}
\label{paths}
\caption{Integration paths.}
\end{figure}
We introduce also a 1-form $\W_{\gb',\gb''}$, called {\em the phase form}, defined 
as follows:
\ben
\W_{\gb',\gb''} = -\(I_{\gb'}^{(0)}(\tau,\Lambda),I_{\gb''}^{(0)}(\tau,\Lambda) \)d\Lambda,\quad
\gb',\gb''\in H_0(f_\tau^{-1}(\Lambda);\QQ).
\een

\begin{lemmaa}
The vertex operators $\Gamma_\tau^{\gb_\pm}$ and $\Gamma_\tau^{-\gb_\pm}$ factor as follows:
\ben
\Gamma_\tau^{\gb_\pm} = e^{\pm K}\,
\Gamma_\tau^{(\gb_\pm+\gb_\mp)/2} \,
\Gamma_\tau^{\pm(\gb_\pm-\gb_\mp)/2}, \quad
\Gamma_\tau^{-\gb_\pm} = e^{\pm K}\,
\Gamma_\tau^{-(\gb_\pm+\gb_\mp)/2}\, 
\Gamma_\tau^{\mp(\gb_\pm-\gb_\mp)/2},
\een
where
\ben
K=\int_\Lambda^{u_i}
\W_{(\gb_+-\gb_-)/2,(\gb_++\gb_-)/2}.
\een
\end{lemmaa}
\proof This is Proposition 4 from \cite{G2}, section 7. \qed 

\begin{lemmab}
For $\Lambda$ near the critical value $u_i$, the following formula holds: 
\beq\label{2toda_kdv}
\Gamma_\tau^{\pm(\gb_+-\gb_-)/2}\(\Psi R e^{U/z}\)\sphat  = 
e^{(W_i+w_i)/2}\(\Psi R e^{U/z}\)\sphat \Gamma^{\pm},
\eeq
where
\ben
&&
W_i = \lim_{\ge\rightarrow 0} \int_\Lambda^{u_i+\ge}\left(
\W_{(\gb_+-\gb_-)/2,(\gb_+-\gb_-)/2}+
\frac{d\xi}{2(\xi-u_i)}\right)\ ,\
w_i = -\int_{\Lambda-u_i}^\Lambda \frac{d\xi}{2\xi} , \\
&&
\Gamma^\pm = \exp\(\sum_{n\in\Z} (-z\d_\Lambda)^n \frac{e_i}{\pm\sqrt{ 2\Lambda}}\).
\een 
\end{lemmab}
\proof This is  Theorem 3 from \cite{G2}. \qed

The integration path in the definition of $W_i$ is any path connecting $\Lambda$ and $u_i+\ge$ and 
$\ge\rightarrow 0$ in such a way that $u_i+\ge\rightarrow u_i$ along a straight segment. The integration
path in $w_i$ is the straight segment connecting $\Lambda - u_i$ and $\Lambda$.
The various integration paths are depicted on Figure 1.

Using Lemma A and Lemma B we get that the expression in the $\{\  \}$-brackets in 
\eqref{res_xi} is equal to
\ben
&&
\Gamma_\tau^{(\gb_++\gb_-)/2}\tensor \Gamma_\tau^{-(\gb_++\gb_-)/2} 
 \(\Psi R e^{U/z}\)\sphat\tensor\(\Psi R e^{U/z}\)\sphat  \\
&&
\left\{\sum_\pm c_\pm(\tau,\Lambda)\Gamma^\pm_{(i)}\tensor\Gamma^\mp_{(i)}
       \frac{d\Lambda}{\pm\sqrt{\Lambda}}
\right\} 
\prod_{j=1}^{k+m}\D_{\rm pt}(\q^j)\tensor \prod_{j=1}^{k+m}\D_{\rm pt}(\q^j),
\een
where the index $i$ in $ \Gamma^\pm_{(i)}$ is just to emphasize that the vertex operator is 
acting on the $i$-th factor in the product $\prod_{j=1}^{k+m}\D_{\rm pt}(\q^j)$ and the coefficients 
$c_\pm$ are given by the following formula:
\beq\label{c_pm}
\log c_\pm  = 
2\frac{(Qe^t/x_\pm)^m}{\nu_0-\nu_1} - \log (x_\pm^2f_\tau'(x_\pm))+
W_i+w_i\pm 2K + \int_{\gamma_\pm}\frac{d\xi}{2\xi},
\eeq
where the path $\gamma_+$ is the composition of $C\circ C_0$ and the 
line segment from 1 to $u_i+1$ and $\gamma_-=\gamma_+\circ\gamma'$, where 
$\gamma'$ is a simple loop around $0$ starting and ending at $1$ (see Figure 1).

We will prove that with respect to $\Lambda$ the functions $c_+$ and $c_-$ are analytic and  coincide 
in a neighborhood of $u_i$. This would finish the proof of the theorem because, according to 
A. Givental \cite{G2}, the 1-form
\ben
\sum_{\pm }\ \Gamma^\pm_{(i)}\tensor\Gamma^\mp_{(i)}\ 
       \frac{d\Lambda}{\pm\sqrt{\Lambda}}\  \T\tensor\T
\een
is analytic in $\Lambda$ whenever $\T$ is a tau-function of the KdV hierarchy. On the other hand,
according to M. Kontsevich \cite{Ko}, $\D_{\rm pt}$ is a tau-function of the KdV hierarchy, thus the
theorem would follow. 

Note that the first two terms in \eqref{c_pm}, up to a summand of 
$\Lambda/(\nu_0-\nu_1)$, coincide with the primitive (see \eqref{pr}) of the 1-form 
$\W_{\gb_\pm, \gb_\pm}.$ Thus
\ben
2\frac{(Qe^t/x_\pm(\Lambda))^m}{\nu_0-\nu_1} - \log (x_\pm^2f_\tau') =  
\int_{\Lambda_0}^\Lambda 
\W_{\gb_\pm,\gb_\pm} + \Lambda/(\nu_0-\nu_1)+C_\pm 
\een
where the constants $C_\pm$ are independent of $\Lambda$ (they depend only on $x_\pm(\Lambda_0)$). 
and their difference can be interpreted as 
\ben
C_+-C_-=\oint_{\gamma}\W_{\gb_-,\gb_-},
\een
where  $\gamma$ is a simple loop around $u_i$ (see Figure 1). Therefore $\log c_\pm$
admits the following integral presentation
\ben
&&
\log c_\pm = \lim_{\ge\rightarrow 0} \\
&&
\left(
\int_{\Lambda_0}^\Lambda \W_{\gb_\pm,\gb_\pm} + 
\int_{\Lambda}^{u_i+\ge} 
\W_{(\gb_+-\gb_-)/2,(\gb_+-\gb_-)/2}
\pm 2\int_{\Lambda}^{u_i+\ge}
\W_{(\gb_+-\gb_-)/2,(\gb_++\gb_-)/2} + \right.\\
&&
\left.
+\int_\Lambda^{u_i+\ge} \frac{d\xi}{2(\xi-u_i)} - \int_{\Lambda-u_i}^\Lambda\frac{d\xi}{2\xi} + 
\int_{\gamma_\pm}\frac{d\xi}{2\xi} +\frac{1}{\nu_0-\nu_1}\Lambda + C_\pm \right)\ .
\een
In the first integral put $\gb_\pm=(\gb_\pm+\gb_\mp)/2 + (\gb_\pm-\gb_\mp)/2$. 
After a simple computation we get:
\ben
&&
\log c_\pm  = 
\int_{\Lambda_0}^\Lambda 
\W_{(\gb_\pm+\gb_\mp)/2,(\gb_\pm+\gb_\mp)/2} +\frac{1}{\nu_0-\nu_1}\Lambda + C_\pm \\
\notag
&&
\lim_{\ge\rightarrow 0} 
\left(
\int_{\Lambda_0}^{u_i+\ge} 
\W_{(\gb_+-\gb_-)/2,(\gb_+-\gb_-)/2}  
+ 2\int_{\Lambda_0}^{u_i+\ge}
\W_{(\gb_\pm-\gb_\mp)/2,(\gb_\pm+\gb_\mp)/2} + 
\int_{\gamma_\pm'} \frac{d\xi}{2\xi}\right) \ ,
\een
where $\gamma_\pm'$ is the composition of the paths: $\gamma_\pm$ -- starting at 1 and ending at 
$\Lambda$, the straight segment between $\Lambda$ and $\Lambda-u_i$ (i.e., the integration path 
for $w_i$), and the path from $\Lambda-u_i$ to $\ge$ obtained by translating the path between 
$\Lambda$ and $u_i+\ge$.
Furthermore, we rewrite the last formula as follows:
\ben
\int_{u_i}^\Lambda 
\W_{(\gb_\pm+\gb_\mp)/2,(\gb_\pm+\gb_\mp)/2} + \frac{1}{\nu_0-\nu_1}\Lambda + C_\pm + 
\lim_{\ge\rightarrow 0} 
\left(
\int_{\Lambda_0}^{u_i+\ge} 
\W_{\gb_\pm,\gb_\pm} +\int_{\gamma_\pm'} \frac{d\xi}{2\xi}\right) \ .
\een
The first integral is analytic near $\Lambda=u_i$, because near $\Lambda=u_i$, 
the mode $I^{(0)}_{\gb_\pm}$ expands as a Laurent series in $\sqrt{\Lambda-u_i}$ with singular term
at most $1/\sqrt{(\Lambda-u_i)}$. However the analytical continuation around $\Lambda=u_i$ transforms
$I^{(0)}_{\gb_\pm}$ into $I^{(0)}_{\gb_\mp}$, hence $I^{(0)}_{\gb_\pm}+I^{(0)}_{\gb_\pm}$ must be
single-valued and in particular, it could not have singular terms. 
Since the limit is independent of $\Lambda$, the analyticity of $c_\pm$ follows. 

It remains to prove that $c_+$ and $c_-$ are equal. 
\ben
\log c_+-\log c_- = \lim_{\ge\rightarrow 0} \left\{
\oint_{\gamma_\ge} \W_{\gb_-,\gb_-}+
\oint_{\gamma'} \frac{d\xi}{2\xi} \right\},
\een  
where $\gamma_\ge$ is a closed loop around $u_i$ starting and ending at $u_i+\ge$. 
The second integral is 
$\pm\pi i$ (the sign depends on the orientation of the loop $\gamma'$). 
To compute the first one, write $\gb_-=(\gb_--\gb_+)/2 +(\gb_-+\gb_+)/2$ and 
transform  the integrand into
\ben
\(  I_{(\gb_--\gb_+)/2}^{(0)},I_{(\gb_--\gb_+)/2}^{(0)}\) + 
2\(  I_{(\gb_--\gb_+)/2}^{(0)},I_{(\gb_-+\gb_+)/2}^{(0)}\)+
\(  I_{(\gb_-+\gb_+)/2}^{(0)},I_{(\gb_-+\gb_+)/2}^{(0)}\).
\een
The last term does not contribute to the integral because it is analytic in $\Lambda$. 
The middle one, up to a factor analytic in $\Lambda$, coincides with 
$(\Lambda-u_i)^{-1/2}$, therefore its integral along $\gamma_\ge$ vanishes in 
the limit $\ge\rightarrow 0$. 
Finally, the first term has an expansion of the type
\ben
\(  I_{(\gb_--\gb_+)/2}^{(0)},I_{(\gb_--\gb_+)/2}^{(0)}\) = \frac{1}{2(\Lambda - u_i)} + O(\Lambda -u_i)
\een 
and so it contributes only $\pm\pi i$ to the integral. Thus
$(\log c_+ -\log c_-) $ is an integer multiple of $2\pi i$, which implies that $c_+=c_-$.

\appendix
\section{Proof of \coref{dJ}} \label{comb_setup}

\subsection{Combinatorial notations }\label{comb_setup:1} 
We assume that $k,m$ are co-prime. Without loss of generality we may assume that $k>m$.
For each integer $i$ with $1\leq i\leq m-1$ we define two positive integers 
$q_i, r_i$ as follows:
$$ik=mq_i+r_i, \quad \text{where } 0\leq r_i\leq m-1.$$
Note that $r_i\neq 0$, otherwise $k, m$ are not co-prime. Also put $q_0=0, q_m=k$. Clearly we have $$\frac{q_i}{k}<\frac{i}{m}<\frac{q_i+1}{k}.$$

\begin{lemma}
$\frac{i+1}{m}>\frac{q_i+1}{k}$.
\end{lemma}
\begin{proof}
The inequality is equivalent to $r_i+k>m$, which follows from the assumption $k>m$.
\end{proof}  

We introduce the sequence $\{s_\alpha\}_{\ga=1}^{k+m}$ which is a rearrangement of the set\footnote{Note that we treat $\frac{0}{k}$ and $\frac{0}{m}$ as two {\em different} numbers.} of numbers $\{\frac{0}{k}, \frac{1}{k},...,\frac{k-1}{k}, \frac{0}{m}, \frac{1}{m},...,\frac{m-1}{m}\}$ into increasing order: $s_1=\frac{0}{k}$, and 
$$s_\alpha =\left\{
 \begin{array}{rr}
 \frac{j}{m}, \quad &\mbox{if $\alpha=j+2+\sum_{i=0}^jq_i$}\\
 \frac{q_j+l}{k},\quad &\mbox{if $\alpha=j+2+l+\sum_{i=0}^j q_i$}
 \end{array}\right.   $$
where $0\leq j\leq m-1$ and $1\leq l\leq q_{j+1}$.

We define differential operators $\delta_\alpha$ by the following rule:
\begin{itemize}
\item
if $s_\alpha\in \{\frac{0}{k}, \frac{1}{k},...,\frac{k-1}{k}\}$, define 
$$\delta_\alpha=\frac{z}{m}\frac{\partial}{\partial\tau}-\frac{\nu_0}{m}-s_\alpha kz.$$
\item
if $s_\alpha\in \{\frac{0}{m}, \frac{1}{m},...,\frac{m-1}{m}\}$, define
$$\delta_\alpha=\frac{z}{k}\frac{\partial}{\partial \tau}-\frac{\nu_1}{k}-s_\alpha mz.$$ 
\end{itemize}

Now put $D_0=D_1=id$ and for $\alpha\geq 2$, define $$D_\alpha:=Q^{-mks_\alpha}e^{-mks_\alpha\tau}\prod_{\gamma<\alpha}\delta_\gamma.$$

For $s_\alpha\in \{\frac{0}{k}, \frac{1}{k},...,\frac{k-1}{k}\}$, define $\stilde_\alpha:=-ms_\alpha$. In this case we may write $s_\alpha=\frac{q_s-a}{k}$ for some $1\leq s\leq m$ and $0\leq a\leq (q_s-q_{s-1})-1$. We have 
$$0\leq r_s+am\leq r_s+mq_s-mq_{s-1}-m=sk-((s-1)k-r_{s-1})-m=k+r_{s-1}-m<k.$$ Thus the fractional part of $-ms_\alpha=\frac{-sk+r_s+am}{k}$ is $\frac{r_s+am}{k}$.

For $s_\alpha=\frac{s}{m}$, we define $\stilde_\alpha:=-ks_\alpha$. We have $-ks=-mq_s-r_s=-m(1+q_s)+m-r_s$. Thus the fractional part of $-ks_\alpha$ is $\frac{m-r_s}{m}$.

\subsection{Proof of \coref{dJ}}
It is straightforward to see that the vector $D_\alpha J_X$ is a linear combination of $z$ times partial derivatives of $J_X$ (restricted to $H^2(\CC_{k,m})$). We will prove the following equalities:

\begin{equation}\label{0-m}
\delta_1J_X=(mg_{m/m})^{-1}z\partial_{m/m}J_X.
\end{equation}

\begin{equation}\label{0-k}
\delta_2J_X=(kg_{k/k})^{-1}z\partial_{k/k}J_X.
\end{equation}

\begin{equation}\label{other_derivative}
D_\alpha J_X=z\partial_{\stilde_\alpha}J_X, \quad \alpha \geq 3.
\end{equation}

The proof of (\ref{0-m})--(\ref{other_derivative}) requires explicit computations of the left-hand sides of them. (\ref{dJ_ik})--(\ref{dJ_jm}) will follow as a by-product.

First we show (\ref{0-m}). Applying $\delta_1$ to (\ref{J-func}) yields 
\begin{equation}\label{cal_d1}
\begin{split}
\delta_1J_X&=ze^{\tau\nu_0/z}\sum_{d>0}\frac{Q^{dm}e^{dm\tau}(\frac{z}{m}(\frac{\nu_0}{z}+dm)-\frac{\nu_0}{m})}{d!z^d\prod_{b=\{dm/k\}}^{dm/k}(\nu+bz)}{\bf 1}_{-dm/k}\\
&\quad+ze^{\tau\nu_1/z}(\frac{z}{m}\frac{\nu_1}{z}-\frac{\nu_0}{m}){\bf 1}_{0/m}+ze^{\tau\nu_1/z}\sum_{d>0}\frac{Q^{dk}e^{dk\tau}((\frac{\nu_1}{z}+dk)\frac{z}{m}-\frac{\nu_0}{m})}{\prod_{b=\{dk/m\}}^{dk/m}(\bnu+bz) d!z^d}{\bf 1}_{-dk/m}.
\end{split}
\end{equation}
Here the term with highest power in $z$ is $z\bnu{\bf 1}_{0/m}$, hence (\ref{0-m}) holds. To see (\ref{dJ_jm}) in the case $j=m$, we rearrange (\ref{cal_d1}) as follows:
\begin{equation*}
\begin{split}
&\sum_{d>0}\frac{Q^{dm}e^{dm\tau}(\frac{z}{m}(\frac{\nu_0}{z}+dm)-\frac{\nu_0}{m})}{d!z^d\prod_{b=\{dm/k\}}^{dm/k}(\nu+bz)}{\bf 1}_{-dm/k}=\sum_{d>0}\frac{Q^{dm}e^{dm\tau}dz}{d!z^d\prod_{b=\{dm/k\}}^{dm/k}(\nu+bz)}{\bf 1}_{-dm/k}\\
&=\sum_{d\geq 0}\frac{Q^{dm+m}e^{dm\tau+m\tau}}{d!z^d\prod_{b=\{\frac{dm+m}{k}\}}^{\frac{dm+m}{k}}(\nu+bz)}{\bf 1}_{-(dm+m)/k}, \quad (\text{re-indexing}),\\
&(\frac{z}{m}\frac{\nu_1}{z}-\frac{\nu_0}{m}){\bf 1}_{0/m}+\sum_{d>0}\frac{Q^{dk}e^{dk\tau}((\frac{\nu_1}{z}+dk)\frac{z}{m}-\frac{\nu_0}{m})}{\prod_{b=\{dk/m\}}^{dk/m}(\bnu+bz) d!z^d}{\bf 1}_{-dk/m}\\
&=\bnu{\bf 1}_{0/m}+\sum_{d>0}\frac{Q^{dk}e^{dk\tau}}{d!z^d}\frac{\prod_{b<\{\frac{dk-m}{m}\}}(\bnu+bz)}{\prod_{b\leq \frac{dk-m}{m}}(\bnu+bz)}{\bf 1}_{-dk/m}\\
&=\sum_{d\geq 0}\frac{Q^{dk}e^{dk\tau}}{d!z^d}\frac{\prod_{b<\{\frac{dk-m}{m}\}}(\bnu+bz)}{\prod_{b\leq \frac{dk-m}{m}}(\bnu+bz)}{\bf 1}_{\frac{-dk+m}{m}}.
\end{split}
\end{equation*}
(Note our convention on the fractional part $\{\,\}$.)

The proof of (\ref{0-k}) and (\ref{dJ_ik}) for the case $i=k$ is similar.

(\ref{other_derivative}), (\ref{dJ_ik}) for $i\neq k$, and (\ref{dJ_jm}) for $j\neq m$ will be proven together by induction on $\alpha\geq 3$.

{\bf Case $\alpha=3$:} We compute
\begin{equation}\label{cal_D3}
\begin{split}
\delta_2\delta_1J_X&=z\sum_{d>0}\frac{e^{\tau\nu_0/z}Q^{dm}e^{dm\tau}dz(\frac{z}{k}(\frac{\nu_0}{z}+dm)-\frac{\nu_1}{k})}{d!z^d\prod_{b=\{dm/k\}}^{dm/k}(\nu+bz)}{\bf 1}_{-dm/k}\\
&\quad+ze^{\tau\nu_1/z}(\frac{z}{k}\frac{\nu_1}{z}-\frac{\nu_1}{k})\bnu{\bf 1}_{0/m}+z\sum_{d>0}\frac{e^{\tau\nu_1/z}Q^{dk}e^{dk\tau}(\frac{z}{k}(\frac{\nu_1}{z}+dk)-\frac{\nu_1}{k})(\bnu+\frac{dk}{m}z)}{\prod_{b=\{dk/m\}}^{dk/m}(\bnu+bz)d!z^d}{\bf 1}_{-dk/m}\\
&=ze^{\tau\nu_0/z}\sum_{d>0}\frac{Q^{dm}e^{dm\tau}dz(\nu+\frac{dm}{k}z)}{d!z^d\prod_{b=\{dm/k\}}^{dm/k}(\nu+bz)}{\bf 1}_{-dm/k}\\
&\quad+ze^{\tau\nu_1/z}\sum_{d>0}\frac{Q^{dk}e^{dk\tau}(\bnu+\frac{dk}{m}z)dz}{\prod_{b=\{dk/m\}}^{dk/m}(\bnu+bz)d!z^d}{\bf 1}_{-dk/m}
\end{split}
\end{equation}
Here it is easy to see that the term having the highest power of $z$ is $Q^me^{m\tau}{\bf 1}_{-m/k}$. So $Q^{-m}e^{-m\tau}\delta_2\delta_1 J_X=z\partial_{\frac{k-m}{k}}J_X$, proving the case $\alpha=3$ of (\ref{other_derivative}) (note that $\stilde_3=\frac{k-m}{k}$). We can further simplify (\ref{cal_D3}) as follows:

\begin{equation*}
\begin{split}
Q^{-m}e^{-m\tau}\delta_2\delta_1 J_X&=ze^{\tau\nu_0/z}\sum_{d>0}\frac{Q^{dm-m}e^{dm\tau-m\tau}(\nu+\frac{dm}{k}z)}{(d-1)!z^{d-1}\prod_{b=\{dm/k\}}^{dm/k}(\nu+bz)}{\bf 1}_{-dm/k}\\
&\quad+ze^{\tau\nu_1/z}\sum_{d>0}\frac{Q^{dk-m}e^{dk\tau-m\tau}(\bnu+\frac{dk}{m}z)}{\prod_{b=\{dk/m\}}^{dk/m}(\bnu+bz)(d-1)!z^{d-1}}{\bf 1}_{-dk/m}\\
&=ze^{\tau\nu_0/z}\sum_{d\geq 0}\frac{Q^{dm}e^{dm\tau}(\nu+\frac{(d+1)m}{k}z)}{d!z^d\prod_{b=\{(d+1)m/k\}}^{(d+1)m/k}(\nu+bz)}{\bf 1}_{-(d+1)m/k}\\
&\quad+ze^{\tau\nu_1/z}\sum_{d\geq 0}\frac{Q^{dk+(k-m)}e^{dk\tau+(k-m)\tau}(\bnu+\frac{(d+1)k}{m}z)}{\prod_{b=\{(d+1)k/m\}}^{(d+1)k/m}(\bnu+bz)(d-1)!z^{d-1}}{\bf 1}_{-(d+1)k/m}\\
&=ze^{\tau\nu_0/z}\sum_{d\geq 0}\frac{Q^{dm}e^{dm\tau}}{d!z^d}\frac{\prod_{b<\{\frac{dm-(k-m)}{k}\}}(\nu+bz)}{\prod_{b\leq \frac{dm-(k-m)}{k}}(\nu+bz)}{\bf 1}_{\frac{-dm+(k-m)}{k}}\\
&\quad+ze^{\tau\nu_1/z}\sum_{d\geq 0}\frac{Q^{dk+(k-m)}e^{dk\tau+(k-m)\tau}}{\prod_{b=\{\frac{dk+(k-m)}{m}\}}^{\frac{dk+(k-m)}{m}}(\bnu+bz)(d-1)!z^{d-1}}{\bf 1}_{-\frac{dk+(k-m)}{m}}.
\end{split}
\end{equation*}
This is exactly the $i=k-m$ case of (\ref{dJ_ik}), as desired.

{\bf Induction step:} Now consider $3\leq \alpha\leq k+m-1$, suppose that (\ref{other_derivative}) and the corresponding (\ref{dJ_ik}) or (\ref{dJ_jm}) hold for $\alpha$. Note that 
\begin{equation}\label{increasing_alpha}
D_{\alpha+1}J_X=Q^{-kms_{\alpha+1}}e^{-kms_{\alpha+1}\tau}\delta_\alpha\left(Q^{kms_\alpha}e^{kms_\alpha\tau}z\partial_{\stilde_\alpha}J_X\right).
\end{equation}

There are three cases which we handle separately.

{\bf Case 1:} $s_\alpha, s_{\alpha+1}\in \{\frac{0}{k},...,\frac{k-1}{k}\}$.\\
We have $s_{\alpha+1}=s_\alpha+\frac{1}{k}$. We may write $\stilde_\alpha=i/k$ for some $1\leq i\leq k$. According to our discussion at the end of subsection \ref{comb_setup:1} we have $i>m$ and $\stilde_{\alpha+1}=\frac{i-m}{k}$. Also, $\delta_\alpha=\frac{z}{m}\frac{\partial}{\partial\tau}-\frac{\nu_0}{m}-s_\alpha kz$. By induction, (\ref{dJ_ik}) holds for this $i$. We will prove (\ref{other_derivative}) for $\alpha+1$ and (\ref{dJ_ik}) for $i-m$.

Using (\ref{dJ_ik}) for this $i$ we calculate 
\begin{equation}\label{cal_Da_1}
\begin{split}
&\delta_\alpha\left(Q^{kms_\alpha}e^{kms_\alpha\tau}z\partial_{\stilde_\alpha}J_X\right)\\
&=ze^{\tau\nu_0/z}\sum_{d> 0}\frac{Q^{dm+kms_\alpha}e^{dm\tau+kms_\alpha\tau}}{d!z^d}dz\frac{\prod_{b< \{\frac{dm-i}{k}\}}(\nu+bz)}{\prod_{b\leq \frac{dm-i}{k}}(\nu+bz)}{\bf 1}_{\frac{-dm+i}{k}}\\
&\quad+ze^{\tau\nu_1/z}\sum_{d\geq 0}\frac{Q^{dk+i+kms_\alpha}e^{(dk+i)\tau+kms_\alpha\tau}}{\prod_{b=\{\frac{dk+i}{m}\}}^{\frac{dk+i}{m}}(\bnu+bz) d!z^d}(\bnu+\frac{dk+i}{m}z){\bf 1}_{\frac{-(dk+i)}{m}}.
\end{split}
\end{equation}
Here the term having the highest power in $z$ is $$Q^{m+kms_\alpha}e^{m\tau+kms_\alpha\tau}\frac{\prod_{b< \{\frac{m-i}{k}\}}(\nu+bz)}{\prod_{b\leq \frac{m-i}{k}}(\nu+bz)}{\bf 1}_{\frac{-m+i}{k}}=Q^{m+kms_\alpha}e^{m\tau+kms_\alpha\tau}{\bf 1}_{\frac{-m+i}{k}},$$
because $\frac{m-i}{k}>-1$. In view of $s_{\alpha+1}=s_\alpha+\frac{1}{k}$ and (\ref{increasing_alpha}) this implies (\ref{other_derivative}) for $\alpha+1$. Moreover, we may further simplify (\ref{cal_Da_1}) to obtain: 
\begin{equation*}
\begin{split}
&ze^{\tau\nu_0/z}\sum_{d\geq 0}\frac{Q^{dm+m+kms_\alpha}e^{dm\tau+(m+kms_\alpha)\tau}}{d!z^d}\frac{\prod_{b< \{\frac{dm+m-i}{k}\}}(\nu+bz)}{\prod_{b\leq \frac{dm+m-i}{k}}(\nu+bz)}{\bf 1}_{\frac{-dm-m+i}{k}}\\
&+ze^{\tau\nu_1/z}\sum_{d\geq 0}\frac{Q^{dk+i+kms_\alpha}e^{(dk+i+kms_\alpha)\tau}}{d!z^d}\frac{1}{\prod_{b=\{\frac{dk+(i-m)}{m}\}}^{\frac{dk+(i-m)}{m}}(\bnu+bz)}{\bf 1}_{\frac{-(dk+(i-m))}{m}}.
\end{split}
\end{equation*}
Using $s_{\alpha+1}=s_\alpha+\frac{1}{k}$ and removing the factor $Q^{kms_{\alpha+1}}e^{kms_\alpha\tau}$, we obtain (\ref{dJ_ik}) for $i-m$, as desired.

{\bf Case 2:} $s_\alpha=\frac{q_s}{k}, s_{\alpha+1}=\frac{s}{m}$.\\
In this case $\stilde_\alpha=\frac{r_s}{k}, \stilde_{\alpha+1}=\frac{m-r_s}{m}$. Also, $\delta_\alpha=\frac{z}{m}\frac{\partial}{\partial\tau}-\frac{\nu_0}{m}-s_\alpha kz$. By induction, (\ref{dJ_ik}) holds for $i=r_s$. We will prove (\ref{other_derivative}) for $\alpha+1$ and (\ref{dJ_jm}) for $j=m-r_s$.

Using (\ref{dJ_ik}) for $i=r_s$, a similar calculation gives (\ref{cal_Da_1}) with $i$ replaced by $r_s$. In the first sum, the term having the highest power of $z$ is 
$$zQ^{m+kms_\alpha}e^{m\tau+kms_\alpha\tau}\frac{\prod_{b<\{\frac{m-r_s}{k}\}} (\nu+bz)}{\prod_{b\leq \frac{m-r_s}{k}}(\nu+bz)}{\bf 1}_{\frac{r_s-m}{k}}=O(1),$$
because $\frac{m-r_s}{k}=\{\frac{m-r_s}{k}\}$. In the second sum, the term having the highest power of $z$ is 
$$zQ^{r_s+kms_\alpha}e^{r_s\tau+kms_\alpha\tau}\frac{(\bnu+\frac{r_s}{m}z)}{\prod_{b=\{\frac{r_s}{m}\}}^{\frac{r_s}{m}}(\bnu+bz)}{\bf 1}_{-r_s/m}=zQ^{r_s+kms_\alpha}e^{r_s\tau+kms_\alpha\tau}{\bf 1}_{\frac{m-r_s}{m}},$$
because $0<\frac{r_s}{m}<1$. Note that $$r_s+kms_\alpha=r_s+mq_s=ks=kms_{\alpha+1}.$$
We conclude that (\ref{other_derivative}) holds for $\alpha+1$. Further simplifying (\ref{cal_Da_1}) for $i=r_s$ yields
\begin{equation*}
\begin{split}
&ze^{\tau\nu_0/z}\sum_{d\geq 0}\frac{Q^{dm+m+kms_\alpha}e^{dm\tau+(m+kms_\alpha)\tau}}{d!z^d}\frac{1}{\prod_{b=\{\frac{dm+m-r_s}{k}\}}^{\frac{dm+m-r_s}{k}}(\nu+bz)}{\bf 1}_{\frac{-(dm+(m-r_s))}{k}}\\
&+ze^{\tau\nu_1/z}\sum_{d\geq 0}\frac{Q^{dk+r_s+kms_\alpha}e^{(dk+r_s+kms_\alpha)\tau}}{d!z^d}\frac{\prod_{b<\{\frac{dk-(m-r_s)}{m}\}}(\bnu+bz)}{\prod_{b\leq \frac{dk-(m-r_s)}{m}}(\bnu+bz)}{\bf 1}_{\frac{-dk+(m-r_s))}{m}},
\end{split}
\end{equation*}
which in turn yields (\ref{dJ_jm}) for $j=m-r_s$ after removing the factor $Q^{sk}e^{sk\tau}$.

{\bf Case 3:} $s_\alpha=\frac{s}{m}, s_{\alpha+1}=\frac{q_s+1}{k}$.\\
In this case $\stilde_\alpha=\frac{m-r_s}{m}, \stilde_{\alpha+1}=\frac{k-m+r_s}{k}$. Also, $\delta_\alpha=\frac{z}{k}\frac{\partial}{\partial \tau}-\frac{\nu_1}{k}-s_\alpha mz$. By induction, (\ref{dJ_jm}) for $j=m-r_s$ holds. We will prove (\ref{other_derivative}) for $\alpha+1$ and (\ref{dJ_ik}) for $i=k-m+r_s$.

Using (\ref{dJ_jm}) for $j=m-r_s$ we calculate 
\begin{equation}\label{cal_Da_2}
\begin{split}
&\delta_\alpha\left(Q^{kms_\alpha}e^{kms_\alpha\tau}z\partial_{\stilde_\alpha}J_X\right)\\
&=ze^{\tau\nu_0/z}\sum_{d\geq 0}\frac{Q^{dm+m-r_s+kms_\alpha}e^{(dm+m-r_s+kms_\alpha)\tau}(\nu+\frac{dm+m-r_s}{k}z)}{d!z^d\prod_{b=\{\frac{dm+m-r_s}{k}\}}^{\frac{dm+m-r_s}{k}}(\nu+bz)}{\bf 1}_{-\frac{(dm+m-r_s)}{k}}\\
&\quad+ze^{\tau\nu_1/z}\sum_{d>0}\frac{Q^{dk+kms_\alpha}e^{(dk+kms_\alpha)\tau}}{(d-1)!z^{d-1}}\frac{\prod_{b<\{\frac{dk-(m-r_s)}{m}\}}(\bnu+bz)}{\prod_{b\leq\frac{dk-(m-r_s)}{m}}(\bnu+bz)}{\bf 1}_{\frac{-dk+m-r_s}{m}}.
\end{split}
\end{equation}
In the first sum, the term having the highest power in $z$ is 
$$z\frac{Q^{m-r_s+kms_\alpha}e^{(m-r_s+kms_\alpha)\tau}(\nu+\frac{m-r_s}{k}z)}{\prod_{b=\{\frac{m-r_s}{k}\}}^{\frac{m-r_s}{k}}(\nu+bz)}{\bf 1}_{-\frac{m-r_s}{k}}=zQ^{m(q_s+1)}e^{m(q_s+1)\tau}{\bf 1}_{\frac{k-m+r_s}{k}},$$
because $0<\frac{m-r_s}{k}<1$ and $m-r_s+kms_\alpha=m(q_s+1)=kms_{\alpha+1}$. In the second sum, the term having the highest power in $z$ is 
$$zQ^{k+kms_\alpha}e^{(k+kms_\alpha)\tau}\frac{\prod_{b<\{\frac{k-m+r_s}{m}\}}(\bnu+bz)}{\prod_{b\leq \frac{k-m+r_s}{m}}(\bnu+bz)}{\bf 1}_{\frac{-k+m-r_s}{m}}=O(1),$$
because $\frac{k-m+r_s}{m}>\frac{r_s}{m}>0$. We conclude that (\ref{other_derivative}) holds for $\alpha+1$. Further simplifying (\ref{cal_Da_2}) yields 
\begin{equation*}
\begin{split}
&ze^{\tau\nu_0/z}\sum_{d\geq 0}\frac{Q^{dm+kms_{\alpha+1}}e^{dm\tau+kms_{\alpha+1}\tau}}{d!z^d}\frac{\prod_{b< \{\frac{dm-(k-m+r_s)}{k}\}}(\nu+bz)}{\prod_{b\leq \frac{dm-(k-m+r_s)}{k}}(\nu+bz)}{\bf 1}_{\frac{-dm+(k-m+r_s)}{k}}\\
&+ze^{\tau\nu_1/z}\sum_{d\geq 0}\frac{Q^{dk+k+kms_\alpha}e^{(dk+k+kms_\alpha)\tau}}{d!z^d}\frac{1}{\prod_{b=\{\frac{dk+(k-m+r_s)}{m}\}}^{\frac{dk+(k-m+r_s)}{m}}(\bnu+bz)}{\bf 1}_{\frac{-(dk+(k-m+r_s))}{m}},
\end{split}
\end{equation*}
which is easily seen to yield (\ref{dJ_ik}) for $i=k-m+r_s$, after using $k+kms_\alpha=kms_{\alpha+1}+(k-m+r_s)$ and removing the factor $Q^{kms_{\alpha+1}}e^{kms_{\alpha+1}\tau}$.
This completes the induction, and the proof of the Corollary.

\section{The bi-graded equivariant reduction of the 2-Toda hierarchy}
\label{2toda}

The 2-Toda lattice hierarchy was introduced by K. Ueno and K. Takasaki 
\cite{UT}. For the purpose of Gromov-Witten theory it is more convenient to introduce a 
hierarchy, which we also call 2-Toda, obtained from the 2-Toda lattice
hierarchy by a certain infinitesimal lattice spacing limiting procedure
(see \cite{ETH}). From now on when we say 2-Toda we always mean the second one, not the original one. 

\subsection{Background on the 2-Toda hierarchy}
The 2-Toda hierarchy  consists of two sequences of flows 
on the manifold of pairs of Lax operators:
\beq\label{lax_ops}
L = \Lambda + \sum_{i\geq 0} a_i \Lambda^{-i} \quad \mbox{ and }\quad
\overline{L} = Qe^v \Lambda^{-1} + \sum_{i\geq 0} \overline{a}_i\Lambda^i,
\eeq
where $Q$ is a fixed constant, $a_i,$ $\overline{a}_j$, $v$ are 
formal series in $\ge,$ whose coefficients are infinitely differentiable
functions, $v$ has no free term: $v=v^1(x)\ge + v^2(x)\ge^2+\ldots $, 
and $\Lambda$ is a formal symbol which secretly should be thought
as the shift operator $e^{\ge \d_x},$ i.e., we demand that
$\Lambda$ and $u(x;\ge)$ satisfy the following commutation relation 
$\Lambda u(x;\ge) = u(x+\ge;\ge)\Lambda:=\Big(\sum_{k\geq 0} 
\frac{1}{k!}\ge^k\d_x^k u(x;\ge)\Big) \Lambda.$

The flows are defined by Lax type equations:
\beqa \label{2toda_1}
\ge \d_{ y_n} L = [\(L^{n}\)_+, L], &
\ge \d_{ y_n} \overline{L} = [\(L^{n}\)_+, \overline{L}],
\quad n\geq 1, \\
\label{2toda_2}
\ge \d_{\overline{ y}_n} L = -[\(\overline{L}^{\ n}\)_-, L], &
\ge \d_{\overline{ y}_n} \overline{L} = 
-[\(\overline{L}^{\ n}\)_-, \overline{L}\ ],\quad n\geq 1,
\eeqa
where if $M$ is a formal series in $\Lambda$ and $\Lambda^{- 1}$ then
we denote by $M_+$ (resp. $ M_-$ ) the series obtained from $M$ by 
truncating the terms with negative (resp. non-negative) powers
of $\Lambda.$

\medskip

Given a pair of Lax operators \eqref{lax_ops} we say that 
\ben
\P=1+w_1(x;\ge)\Lambda^{-1}+w_2(x;\ge)\Lambda^{-2}+\ldots
\een
and 
\ben
\Q = \overline{w}_0+\overline{w}_1 (\Lambda/Q) + 
\overline{w}_2(\Lambda/Q)^2 +\ldots 
\een
form a pair of {\em dressing operators} if $L = \P\Lambda\P^{-1}$ and 
$\overline{L} = \Q\, Q\Lambda^{-1}\Q^{-1}.$ According to \cite{UT}, 
Proposition 1.4, the pair of Lax operators $L$ and $\overline{L}$ is a
solution to the 2-Toda hierarchy if and only if there is a pair of dressing
operators $\P$ and $\Q$, called {\em wave operators}, such that
\beqa\label{flow1_wave}
\ge \d_{y_n}\,\P = -(L^n)_-\P, & \ge \d_{y_n}\,\Q = (L^n)_+\Q, &  \\
\label{flow2_wave}
\ge \d_{\overline{y}_n}\,\P = -(\overline{L}^n)_-\P, & 
\ge \d_{\overline{y}_n}\,\Q = (\overline{L}^n)_+\Q, & n\geq 1. 
\eeqa
Let us remark that the two sequences of time variables in \cite{UT}, denoted 
there by $x_n$ and $y_n$, correspond in our notations 
respectively to $y_n/\ge$ and $-\overline{y}_n/\ge$. The reason
for the negative sign is that our definition of the flows \eqref{2toda_2}  
differs from the one in \cite{UT} by a negative sign.

Given a non-zero function $\tau(x,\y,\overline{\y};\ge),$ where 
$\y=(y_1,y_2,\ldots)$ and $\overline{\y}=(\overline{y}_1,\overline{y}_2,\ldots)$,
we define two operators $\P=1+w_1\Lambda^{-1}+w_2\Lambda^{-2}+\ldots$ and 
$\Q=\overline{w}_0+\overline{w}_1(\Lambda/Q)+\overline{w}_2(\Lambda/Q)^{2}+\ldots,$ by  
\beq\label{def_tau_p}
1+w_1\gl^{-1}+w_2\gl^{-2} +\ldots = 
\frac{ \exp\Big(
-\sum_{n=1}^\infty \frac{\gl^{-n}}{n}\,\ge\d_{y_n}     \Big)
\tau(x,\y,\overline{\y};\ge) }{\tau(x,\y,\overline{\y};\ge)}
\eeq
and 
\beq\label{def_tau_q}
\overline{w}_0+\overline{w}_1\gl^{-1}+\overline{w}_2\gl^{-2}+\ldots =
\frac{ \exp\Big(
\sum_{n=1}^\infty \frac{\gl^{-n}}{n}\,\ge\d_{\overline{y}_n}     \Big)
\tau(x+\ge,\y,\overline{\y};\ge) }{\tau(x,\y,\overline{\y};\ge)}.
\eeq
The function $\tau(x,\y,\overline{\y};\ge)$ is called $\tau$-{\em function} of
the 2-Toda hierarchy if the corresponding operators $\P$ and $\Q$ form a 
pair of wave operators, i.e., they satisfy equations 
\eqref{flow1_wave}--\eqref{flow2_wave}.

Let us remark that our definitions of wave operators and $\tau$-functions 
are slightly different from the ones in \cite{UT}. 
Namely, we define the wave operator $\Q$ via the identity 
$\overline{L}=\Q\(Q\Lambda^{-1}\)\Q^{-1},$ while in \cite{UT} the definition
is $\overline{L}=\Q'\Lambda^{-1}(\Q')^{-1}.$ On the other hand  
$Q\Lambda^{-1}= Q^{x/\ge}\Lambda^{-1}Q^{-x/\ge}$, therefore 
$\Q'=\Q Q^{x/\ge}.$  
Our excuse for departing from the standard definition is that we prefer to 
work with wave operators that admit a quasi-classical limit 
$\ge\rightarrow 0.$ Note that if we put 
$\Q'=\overline{w}'_0 +\overline{w}'_1 \Lambda +\overline{w}_2'\Lambda+\ldots$ 
and 
$\Q= \overline{w}_0  +\overline{w}_1(\Lambda/Q)+
\overline{w}_2(\Lambda/Q)^{2}+\ldots,$ then 
$\overline{w}_i'=\overline{w}_i Q^{x/\ge}.$  This implies that if we 
define $\tau'$ the same way as $\tau$ except that in \eqref{def_tau_q}
we use $\overline{w}_i'$ instead of $\overline{w}_i$ then 
$\tau'= Q^{\frac{1}{2}\( (x/\ge)^2-(x/\ge)\)}\tau.$ 

Let us introduce the following vertex operators:
\ben
\Gamma^\pm = \exp \Big(\pm \sum_{n=1}^\infty (y_n/\ge)\gl^n \Big)
\exp \Big(\mp \sum_{n=1}^\infty \frac{\gl^{-n}}{n}\,\ge\d_{y_n} \Big)
\een
and $\overline{\Gamma}^\pm$ defined by the same formulas as $\Gamma^\pm$ 
but with $\overline{y}_n$ instead of $y_n.$ Then according to \cite{UT},
Theorem 1.7 and Proposition 1.6, the Lax operators $L=\P\Lambda\P^{-1}$
and $\overline{L}=\Q'\Lambda^{-1}(\Q')^{-1}$ form a solution of the 
2-Toda hierarchy iff $\tau'$ satisfies the following HQEs: 
\ben
\res_{\gl=\infty}\frac{d\gl}{\gl}
\Big( \gl^{l-n}\, (\Gamma^+\, \tau'_l    )\tensor
                (\Gamma^-\, \tau'_{n+1}) - 
\gl^{n-l}\, 
(\overline{\Gamma}^{\,-} \tau'_{l+1} )\tensor
(\overline{\Gamma}^{\,+} \tau'_n     )\Big) = 0,
\een 
where for every integer $r$ we put 
$\tau_r':=\tau'(x+r\ge,\y,\overline{\y};\ge).$  Substituting in the above
HQEs the formula for $\tau'$ in terms of $\tau$ we get that 
$\tau(x,\y,\overline{\y};\ge)$ is a $\tau$-function iff the following 
HQEs hold:
\beq\label{HQE_Q}
\res_{\gl=\infty}\frac{d\gl}{\gl}
\Big( \gl^{l-n}\, (\Gamma^+\, \tau_l    )\tensor
                (\Gamma^-\, \tau_{n+1}) - 
(Q\gl^{-1})^{l-n}\, 
(\overline{\Gamma}^{\,-} \tau_{l+1} )\tensor
(\overline{\Gamma}^{\,+} \tau_n     )\Big) = 0.
\eeq
\subsection{The equivariant bi-graded reduction}
According to the change of variables  \eqref{change_qy} and \eqref{change_qby}
we have $q_0^{0/k} =(\nu_0-\nu_1)y_k$ and 
$q_0^{0/m} =(\nu_1-\nu_0)\overline{y}_m.$ Note that the shift of 
$q_0^{0/k}$ (resp. $q_0^{0/m}$) by $n\ge $ is equivalent to
shifting $y_k$ (resp. $\overline{y}_m$) by $\frac{n\ge}{\nu_0-\nu_1}$
(resp. $\frac{n\ge}{\nu_1-\nu_0}$). 
Motivated by \thref{t2} we ask the following
\begin{question}
What are the solutions $L$ and $\overline{L}$ of the 2-Toda hierarchy such that 
the corresponding $\tau$-function has the form 
\beq\label{tau_reduction}
\tau(x,\y,\overline{\y};\ge)= 
\D(y_1,\ldots,y_k + \frac{x}{\nu_0-\nu_1},\ldots ,
\overline{y}_1,\ldots,\overline{y}_m + \frac{x}{\nu_1-\nu_0},\ldots; \ge),   
\eeq
i.e., $(\nu_0-\nu_1)\d_x \tau = (\d_{y_k}-\d_{\overline{y}_m})\tau?$
\end{question}
This is equivalent to the following conditions on wave operators:
\beq\label{reduction_wave}
(\nu_0-\nu_1)\d_x \P = (\d_{y_k}-\d_{\overline{y}_m})\P 
\quad \mbox{ and } \quad
(\nu_0-\nu_1)\d_x \Q = (\d_{y_k}-\d_{\overline{y}_m})\Q .
\eeq
We define the logarithms of the Lax operators $L$ and $\overline{L}$ by 
\ben
\log L := \P \log \Lambda \P^{-1} := \ge\d_x -\(\ge\d_x\P\)\P^{-1} 
\een
and 
\ben
\log \overline{L} := \Q \log\(Q\Lambda^{-1}\)\Q^{-1} :=
-\ge\d_x+\log Q +\(\ge\d_x\Q\)\Q^{-1}.   
\een
On the other hand from equations \eqref{reduction_wave} we get
\ben
\ge\d_x\P=\frac{1}{\nu_0-\nu_1}\(\ge\d_{y_k}\P-\ge\d_{\overline{y}_m}\P\)=
\frac{1}{\nu_0-\nu_1}\(-(L^k)_-\P + (\overline{L}^{\,m})_-\P\)
\een
and 
\ben
\ge\d_x\Q=\frac{1}{\nu_0-\nu_1}\(\ge\d_{y_k}\Q-\ge\d_{\overline{y}_m}\Q\)=
\frac{1}{\nu_0-\nu_1}\((L^k)_+\Q - (\overline{L}^{\,m})_+\Q\).
\een
Using our definition of logarithms of the Lax operators we write the above 
relations in the following form:
\ben
L^k + (\nu_1-\nu_0)\log L = \(L^k\)_+ + \(\overline{L}^m\)_- + (\nu_1-\nu_0)\ge\d_x  
\een
and 
\ben
\overline{L}^m + (\nu_0-\nu_1)\log \(Q^{-1}\overline{L}\) = 
\(L^k\)_+ + \(\overline{L}^m\)_- + (\nu_1-\nu_0)\ge\d_x.
\een
Now the description of the new hierarchy is the following. We define
flows on the manifold of Lax operators  
\ben
\L := \Lambda^k + \sum_{i=1}^k u_i \Lambda^{k-i} + 
\sum_{j=1}^{m-1} u_{k+j} \Lambda^{-j} + \(Qe^v\Lambda^{-1}\)^m 
+(\nu_1-\nu_0)\ge\d_x. 
\een
Note that the equations $L^k + (\nu_1-\nu_0)\log L =\L$ and 
$\overline{L}^m + (\nu_0-\nu_1)\log \(Q^{-1}\overline{L}\)=\L$ have unique
solutions of the types respectively 
$L=\Lambda + a_0 +a_1\Lambda^{-1}+a_2\lambda^{-2}+\ldots$
and $\overline{L} = Qe^v\Lambda^{-1} + \overline{a}_0 + 
\overline{a}_1\Lambda + \overline{a}_2\Lambda^2 + \ldots ,$ where $a_i$
and $\overline{a}_j$ are formal series in $\ge$ whose coefficients are
differential polynomials in $u_1,u_2,\ldots, u_N :=Qe^v.$ The flows of 
the hierarchy are given by:
\beq\label{new_hierarchy}
\ge\d_{y_n} \L = [\(L^n\)_+,\L],\quad 
\ge\d_{\overline{y}_n} \L = -[\(\overline{L}^n\)_-,\L],\quad n\geq 1.
\eeq
One can check easily that this is a commuting set of flows. Also,
by tracing back our argument, one can check that all solutions $\L$ 
are given by 
\ben
\L =\(\P\Lambda^k\P^{-1}\)_+ + \(\Q (Q\Lambda^{-1})^m\Q^{-1}\)_- 
+ (\nu_1-\nu_0)\ge\d_x, 
\een 
where $\P$ and $\Q$ are defined by formulas \eqref{def_tau_p} and 
\eqref{def_tau_q}, for some function $\tau$ of the type \eqref{tau_reduction} 
satisfying the bi-linear identities \eqref{HQE_Q}.

In order to check that we have an integrable hierarchy one needs to find
a Hamiltonian formulation and prove the completeness of the flows. 
This could be done in the same way as in the article \cite{Ge}.  
Another interesting problem is to prove that the Extended Bi-graded 
Toda Hierarchy (EBTH) defined in \cite{C} is a non-equivariant 
limit of our hierarchy (\ref{new_hierarchy}). 

It is shown in \cite{C} that EBTH is bi-hamiltonian, while the methods of E. Getzler \cite{Ge} give only one Hamiltonian structure for (\ref{new_hierarchy}). A natural question is whether the second Hamiltonian structure admits an equivariant deformation. A positive answer to the last question would be an indication that the big project of B. Dubrovin and Y. Zhang \cite{DZ} admits a generalization in the context of equivariant quantum cohomology.


\end{document}